\documentclass[11pt,dvipsnames]{amsart}
\usepackage{amsmath,amssymb,amsfonts}
\usepackage[mathscr]{eucal}
\usepackage[margin=1in]{geometry}

\usepackage{tikz, tikz-cd}
\usetikzlibrary{arrows}
\usetikzlibrary{decorations.pathmorphing}
\usetikzlibrary{decorations.markings}
\usetikzlibrary{shapes}
\usetikzlibrary{backgrounds}
\usetikzlibrary{patterns}
\usepackage[all]{xy}
\usepackage{comment}

\usepackage[pagebackref, colorlinks, urlcolor=gray, citecolor=gray, linkcolor=black]{hyperref}
\hypersetup{linktoc=all,} 

\usepackage[textwidth=25mm, textsize=tiny]{todonotes}
\usepackage[final]{microtype}

\usepackage{enumerate}

\numberwithin{equation}{section} 
\numberwithin{figure}{section}
\usepackage[nameinlink,capitalise,noabbrev]{cleveref}

\newcommand{\newrefformat}[2]{}

\crefname{lemma}{Lemma}{Lemmas}
\crefname{theorem}{Theorem}{Theorems}
\crefname{definition}{Definition}{Definitions}
\crefname{proposition}{Proposition}{Propositions}
\crefname{remark}{Remark}{Remarks}
\crefname{corollary}{Corollary}{Corollaries}
\crefname{equation}{Equation}{Equations}
\crefname{construction}{Construction}{Constructions}
\crefname{ex}{Example}{Examples}
\crefname{appsec}{Appendix}{Appendices}
\crefname{subsection}{Subsection}{Subsections}

\theoremstyle{plain}
\newtheorem{theorem}[equation]{Theorem}
\newtheorem{corollary}[equation]{Corollary}
\newtheorem{proposition}[equation]{Proposition}

\newtheorem{conjecture}[equation]{Conjecture}

\newtheorem{introtheorem}{Theorem}
 
\crefname{introtheorem}{Theorem}{Theorems}

\theoremstyle{definition}
\newtheorem{definition}[equation]{Definition}
\newtheorem{example}[equation]{Example}
\newtheorem{remark}[equation]{Remark}

\newcommand{\RR}{\mathbb{R}}

\newcommand{\NN}{\mathbb{N}}
\newcommand{\ZZ}{\mathbb{Z}}

\newcommand{\PP}{\mathbb{P}}

\newcommand{\into}{\hookrightarrow}

\renewcommand{\phi}{\varphi}
\newcommand{\bndry}{\partial}

\newcommand{\hocolim}{\operatorname{hocolim}}
\renewcommand{\hom}{\operatorname{Hom}}

\newcommand{\id}{\operatorname{id}}

\newcommand{\ob}{\operatorname{Ob}}
\newcommand{\op}{\operatorname{op}}

\newcommand{\FMS}{\operatorname{FMS}}

\newcommand{\dom}{\operatorname{dom}}

\newcommand{\ED}{\operatorname{ED}}
\newcommand{\VD}{\operatorname{VD}}

\newcommand{\abs}[1]{\lvert#1\rvert}
\newcommand{\bcat}[1]{{\bf #1}}
\newcommand{\cat}[1]{\mathscr{#1}}

\newcommand{\grI}{~\begin{tikzpicture}[scale=0.35]
    \fill (0,0) circle (3pt);
    \fill (1,1) circle (3pt);
    
    \draw[thick] (1,1) -- (0,0);
\end{tikzpicture}~}

\newcommand{\grII}{~\begin{tikzpicture}[scale=0.35]
    \fill (1,0) circle (3pt);
    \fill (0,0) circle (3pt);
    \fill (1,1) circle (3pt);
    \fill (0,1) circle (3pt);
    
    \draw[thick] (1,0) -- (0,0);
    \draw[thick] (0,1) -- (1,1);
\end{tikzpicture}~}

\newcommand{\grL}{~\begin{tikzpicture}[scale=0.35]
    \fill (4,0) circle (3pt);
    \fill (3,0) circle (3pt);
    \fill (4,1) circle (3pt);
    \fill (3,1) circle (3pt);
    
    \draw (4,0) -- (3,0);
    \draw (3,0) -- (3,1);
\end{tikzpicture}~}

\newcommand{\grTri}{~\begin{tikzpicture}[scale=0.35]
    \fill (0.75,0) circle (3pt);
    \fill (-0.75,0) circle (3pt);
    \fill (0,1) circle (3pt);
    \fill (0.75,1) circle (3pt);
    \draw (0,1) -- (0.75,0) -- (-0.75,0) -- cycle;
\end{tikzpicture}~}

\newcommand{\grClaw}{~\begin{tikzpicture}[scale=0.35]
    \fill (0.75,0) circle (3pt);
    \fill (-0.75,0) circle (3pt);
    \fill (0,1) circle (3pt);
    \fill (0,0.5) circle (3pt);
    
    \draw (0,1) -- (0,0.5);
    \draw (-0.75, 0) -- (0,0.5) -- (0.75,0);
\end{tikzpicture}~}

\author[M. E. Calle]{Maxine E. Calle}             
\email{callem@sas.upenn.edu}
\address{Department of Mathematics,
         University of Pennsylvania,
         Philadelphia, PA, 19104,
         USA}

\author[J. J. Gould]{Julian J. Gould}
\email{jjgould@sas.upenn.edu}
\address{Department of Mathematics,
         University of Pennsylvania,
         Philadelphia, PA, 19104,
         USA}

\date{\today}



\title[A $K$-Theory Perspective on Edge Reconstruction]{A Combinatorial $K$-Theory Perspective on\\the edge reconstruction conjecture\\in Graph Theory}

\begin{document}
\maketitle
\begin{abstract}
    We provide a framework for abstract reconstruction problems using the $K$-theory of categories with covering families, which we then apply to reformulate the edge reconstruction conjecture in graph theory. Along the way, we state some fundamental theorems for the $K$-theory of categories with covering families.
\end{abstract}

\section{Introduction}\label{sec:intro}

Higher algebraic $K$-theory is a useful tool in many areas of mathematics, including algebraic geometry, number theory, and topology. Inspired by the constructions of Quillen \cite{quillen:73} and Waldhausen \cite{waldhausen:85}, a relatively recent research program pioneered by Campbell and Zakharevich seeks to expand the reach of algebraic $K$-theory to contexts with more of a combinatorial flavor \cite{CamZak, zakharevich:12, zakharevich:16a}. This combinatorial $K$-theory program, while still developing, has sparked a number of ongoing research projects, many of which are related to versions of Hilbert's third problem about scissors congruence of polytopes \cite{zakharevich:12, CamZak, malkiewich} as well as the $K$-theory of varieties \cite{Cam17, CZ18}. 

In \cite{bohmann/gerhardt/malkiewich/merling/zakharevich:23}, Bohmann--Gerhardt--Malkiewich--Merling--Zakharevich introduce a $K$-theory construction for categories with covering families. A category with covering families comes equipped with \textit{multimorphisms} $\{X_i\to X\}_{i\in I}$, comprised of a collection of morphisms $X_i\to X$ indexed over a finite set $I$, which is a ``cover'' of the object $X$. Their $K$-theory machine takes in a category with covering families and outputs a spectrum which decomposes objects in terms of their covers. This construction is a generalization of Zakharevich's assembler $K$-theory \cite{zakharevich:16a}, and in \cref{sec:KT} we adapt some of Zakharevich's theorems for assembler $K$-theory to the setting of categories with covering families (see \cref{thm: equiv of cats on KT}, \cref{thm:devissage}, and \cref{thm:localization}).

This flavor of $K$-theory can be used as a framework for reconstruction problems. A reconstruction problem is a statement of the form \textit{Given some partial information $\{X_i\}_{i\in I}$ about some object $X$, can we determine $X$ up to isomorphism?} When the partial information $\{X_i\}_{i\in I}$ is encoded as a multimorphism $\{X_i\to X\}_{i\in I}$ in some category with covering families, we can apply the machinery of combinatorial $K$-theory. We explore these ideas in more detail in \cref{sec:abstr recon} and formulate an abstract categorical framework for reconstruction problems. Our framework is applicable to problems in a wide variety of fields, including algebra, topology, signal processing, and graph theory.

In this paper, we focus on a particular open reconstruction problem in graph theory, called the edge reconstruction conjecture \cite{harary:1964, kelly42} (see also \cite{bondy/hemminger}). This well-known conjecture states that a finite graph $G$ with at least four edges can be constructed up to isomorphism from a collection of partial information known as an \textit{edge deck}, $\ED(G)$. For example, consider the graph
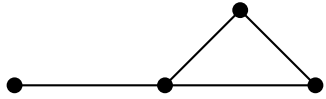
\[
\begin{tikzpicture}
    \fill (-2,0) circle (3pt);
    \fill (0,0) circle (3pt);
    \fill (2,0) circle (3pt);
    \fill (1,1) circle (3pt);
    \draw[thick] (-2,0) -- (0,0) -- (2,0) -- (1,1) -- (0,0);
\end{tikzpicture}
\] and the four subgraphs which can be obtained by deleting a single edge
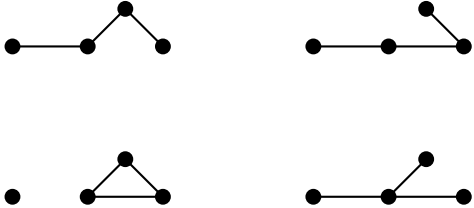
\[
\begin{tikzpicture}
    \fill (-1,0) circle (3pt);
    \fill (0,0) circle (3pt);
    \fill (1,0) circle (3pt);
    \fill (0.5,0.5) circle (3pt);
    \draw[thick] (0,0) -- (1,0) -- (0.5,0.5) -- (0,0);
    
    \fill (-1,2) circle (3pt);
    \fill (0,2) circle (3pt);
    \fill (1,2) circle (3pt);
    \fill (0.5,2.5) circle (3pt);
    \draw[thick] (-1,2) -- (0,2);
    \draw[thick] (1,2) -- (0.5,2.5) -- (0,2);
    
    \fill (3,0) circle (3pt);
    \fill (4,0) circle (3pt);
    \fill (5,0) circle (3pt);
    \fill (4.5,0.5) circle (3pt);
    \draw[thick] (3,0) -- (4,0) -- (5,0);
    \draw[thick] (4.5,0.5) -- (4,0);
    
    \fill (3,2) circle (3pt);
    \fill (4,2) circle (3pt);
    \fill (5,2) circle (3pt);
    \fill (4.5,2.5) circle (3pt);
    \draw[thick] (3,2) -- (4,2) -- (5,2) -- (4.5,2.5);
\end{tikzpicture}
\] Observing that the two subgraphs on the top line are isomorphic, we can record this information in a multiset called the edge deck,\[
\ED\left( \begin{tikzpicture}[scale=0.5]\fill (-1,0) circle (3pt);
    \fill (0,0) circle (3pt);
    \fill (1,0) circle (3pt);
    \fill (0.5,0.5) circle (3pt);
    \draw[thick] (0,0) -- (1,0) -- (0.5,0.5) -- (0,0) -- (-1,0); \end{tikzpicture}\right)
    = 
\left\{\begin{tikzpicture}[scale=0.75]
    \fill (-1,0) circle (3pt);
    \fill (0,0) circle (3pt);
    \fill (1,0) circle (3pt);
    \fill (0.5,0.5) circle (3pt);
    \draw[thick] (0,0) -- (1,0) -- (0.5,0.5) -- (0,0);
    \node at (1.75,0) {$\times 1,$};
    
    \fill (3,0) circle (3pt);
    \fill (4,0) circle (3pt);
    \fill (5,0) circle (3pt);
    \fill (4.5,0.5) circle (3pt);
    \draw[thick] (3,0) -- (4,0) -- (5,0);
    \draw[thick] (4.5,0.5) -- (4,0);
    \node at (5.75,0) {$\times 1,$};
    
    \fill (7,0) circle (3pt);
    \fill (8,0) circle (3pt);
    \fill (9,0) circle (3pt);
    \fill (8.5,0.5) circle (3pt);
    \draw[thick] (7,0) -- (8,0);
    \draw[thick] (9,0) -- (8.5,0.5) -- (8,0);
    \node at (9.75,0) {$\times 2$};
\end{tikzpicture}\right\}.
\] The edge reconstruction conjecture asks whether we can uniquely construct the original graph given just this multiset. The conjecture was originally posed by Harary in 1964 \cite{harary:1964}, inspired by a version posed by Kelly in 1942 \cite{kelly42} which uses vertex-deletion instead of edge-deletion. 
It is known that the vertex reconstruction conjecture implies the edge reconstruction conjecture (see \cite{bondy91}). In particular, if a graph $G$ with $\geq 4$ edges and no isolated vertices can be reconstructed from its vertex-deleted subgraphs, it can be reconstructed from its edge-deleted subgraphs. Indeed, it immediately follows from the Whitney Isomorphism Theorem \cite{whit32} that the edge reconstruction conjecture is true if and only if all line graphs are vertex reconstructable. 

To date, many classes of graphs have been proved to be vertex- or edge-reconstructible. Kelly proved vertex reconstruction for trees in his Ph.D. thesis \cite{kelly42}. 
Wall's masters thesis \cite{wall08} has self-contained proofs of vertex-reconstruction for regular graphs, complete graphs, disconnected graphs, and graphs with disconnected complements. While a proof of vertex reconstruction for planar graphs remains elusive, Giles proved the result for outplanar graphs \cite{giles74}. Since vertex recontruction implies edge reconstruction, these classes of graphs are also edge-reconstructable.
Graphs without an induced subgraph isomorphic to ``the claw" $K_{1,3}$ were shown to be edge-reconstructable by Ellingham, Pyber, and Yu \cite{EPY88}. Kroes provides citations for other classes of edge-reconstructable graphs in terms of minimal/maximal/average vertex degree in his masters thesis \cite{kroes16}. On the other hand, the reconstruction conjecture is known to fail for some generalizations of graphs, such as directed graphs \cite{stockmeyer}, hypergraphs \cite{berge}, and matroids \cite{brylawski, brylawski2}.

In this paper, we reformulate the edge reconstruction conjecture algebraically, using the framework of combinatorial $K$-theory. It is worth noting that we do not prove (or disprove) any part of this conjecture.

\begin{introtheorem}[\cref{defn:gamma for ED}, \cref{cor:VRC iff inj}]
For each $n\geq 1$, there is a category with covering families $\Gamma_{n,n-1}$ so that the edge reconstruction conjecture is true for graphs with $n$ edges if and only if a certain map\[
\iota_n\colon \mathcal G_n\to K_0(\Gamma_{n,n-1})
\] is injective, where $\mathcal{G}_n$ is the set of isomorphism classes of graphs with $n$ edges.
\end{introtheorem}

The idea of approaching edge reconstruction algebraically is not new, see \cite{stanley:84, stanley:85, krasikov/roditty}, but only recent developments have made it possible to approach the problem using $K$-theory. Our results are part of a body of work that seeks to take advantage of these new developments and situate algebraic approaches to geometric ``cut-and-paste'' problems within a $K$-theoretic framework.

For example, Sydler's introduction of the cut-and-paste congruence group for polytopes was crucial to completely classify scissors congruence in 3-dimensions \cite{Syd65}; this approach was extended by Jessen \cite{Jes68} to 4-dimensions and remains an open problem in higher dimensions (see also \cite{Dup01}). Work of Campbell--Zakharevich \cite{zakharevich:12, CamZak} reformulates scissors congruence of polytopes in the context of $K$-theory, and consequently questions about scissors congruence of polytopes can be approached using techniques in $K$-theory and trace methods, as in \cite{bohmann/gerhardt/malkiewich/merling/zakharevich:23}. Similarly, cut-and-paste groups of manifolds (also known as \textit{schneiden und kleben} or $SK$ groups) have been given a $K$-theoretic formulation \cite{hoekzema/merling/murray/rovi/semikina:2021}.

In both cases, these cut-and-paste groups are crucial to classifying \textit{cut-and-paste invariants}. For example, the only cut-and-paste invariants for low-dimensional polytopes are volume and the Dehn invariant (and conjecturally this is also true in higher dimensions); for manifolds, the only cut-and-paste invariants are the Euler characteristic (in $2n$-dimensions) and signature (in $4n$-dimensions). There is also a notion of \textit{(edge) reconstruction invariant} for graphs, which is an invariant $f(G)$ of a graph $G$ so that $f(G)=f(G')$ whenever $\ED(G)= \ED(G')$. For example, the number of isolated vertices, the Euler characteristic, and many polynomials of a graph are edge or vertex reconstruction invariants (see e.g.\ \cite{kotek}). Our construction of $K_0(\Gamma_{n,n-1})$ provides a universal home for these invariants in the following sense.

\begin{introtheorem}[\cref{cor:univ prop reconstr invar}]
    If $f\colon \mathcal{G}_n\to A$ is a reconstruction invariant valued in some Abelian group $A$, then $f$ factors uniquely\[
    \begin{tikzcd}
        \mathcal{G}_n\ar[r, "f"] \ar[d, swap, "\iota_n"] & A\\
        K_0(\Gamma_{n,n-1})\ar[ur, dashed] &
    \end{tikzcd}.
    \]
\end{introtheorem}
A possible avenue for future exploration is to lift known reconstruction invariants to $K$-theory spectra, as the authors in \cite{hoekzema/merling/murray/rovi/semikina:2021} do for the Euler characteristic, although we do not pursue this idea in this paper.

Although we are unable to prove or disprove the edge reconstruction conjecture, using $K$-theory to study this question naturally led us to a disproof of a stronger version of the edge reconstruction conjecture (see \cref{subsec:general ERC}). We note that $K$-theory is certainly not necessary to formulate and prove this result.

\begin{introtheorem}[\cref{thm:inj of i iff GERC}, \cref{thm:GERC false}]
    For every $n\geq 2$, there is some number $k\geq 1$ and some collection of $n$-edge graphs $G_1,\dots, G_k$ which is not reconstructable (up to isomorphism) from the union of the multisets $\ED(G_1),\dots, \ED(G_k)$.
\end{introtheorem} 

Note that for each $n$, there is a minimal $k_n$ for which edge reconstructability fails.
The original edge reconstruction conjecture can be restated as saying that $k_n>1$ for all $n\geq 4$. Our work also opens the door for new combinatorial avenues of exploration. Rather than trying to prove a lower bound on $k_n$, we can instead try to produce upper bounds. At the end of \cref{subsec:general ERC}, we explore some ideas in this direction.

\subsection{Outline}
In \cref{sec:KT}, we recall the $K$-theory construction for categories with covering families and state some fundamental theorems which are direct analogues of theorems for assembler $K$-theory. We give finite graphs the structure of a category with covering families in \cref{sec:graphs}, and state the edge reconstruction conjecture in the language of $K$-theory in the same section (in particular, see \cref{cor:VRC iff inj}). The generalized version is discussed in \cref{subsec:general ERC}. In \cref{sec:future}, we briefly discuss some potential avenues for future research. There is one appendix, \cref{sec:abstr recon}, where we develop an abstract framework for general reconstruction problems, generalizing our methods from edge reconstruction to other settings.

\subsection{Acknowledgements} We are very grateful to Andrew Kwon, as well as Harry Smit, without whom this unlikely collaboration would've never begun. We would also like to thank David Chan and Mona Merling for helpful feedback on early drafts of this paper, as well as the anonymous referee for their comments which improved the exposition in this paper. Additional thanks goes to Anish Chedalavada, Rob Ghrist, Cary Malkiewich, and Maximilien P\'eroux for helpful conversations. The first-named author was partially supported by NSF grant DGE-1845298.
\section{\texorpdfstring{$K$}{TEXT}-theory of categories with covering families}\label{sec:KT}

In this section, we review the definition of categories with covering families from \cite{bohmann/gerhardt/malkiewich/merling/zakharevich:23} and highlight certain results about their $K$-theory groups. The basic idea is that a category with covering families is a category which comes with a notion of ``covers'' of objects, specified by multimorphisms in $\cat C$. Recall that $\{f_i\colon A_i\to A\}_{i\in I}$ is a multimorphism in $\cat C$ if $I$ is a finite (possibly empty) indexing set and for each $i\in I$ the map $f_i\colon A_i\to A$ is a morphism in $\cat C$. The $K$-theory (particularly $K_0$) of a category with covering families decomposes objects into pieces according to these multimorphisms.

\begin{definition}\label{defn:CatFam}
Let $\cat C$ be a small category with a distinguished object $*\in \cat C$ so that $\cat C(*,*) = \{1_*\}$ and $\cat C(c, *)=\varnothing$ for $c\neq *$. Then $\cat C$ is a \textit{category with covering families} if it comes with a collection of multimorphisms, called \textit{covering families}, such that\begin{itemize}
    \item for every finite (possibly empty) indexing set $I$, $\{*=*\}_{i\in I}$ is a covering family;
    \item for all $c\in \cat C$, $\{c=c\}$ is a covering family;
    \item given a covering family $\{g_j\colon B_j\to A\}_{j\in J}$ along with a covering family $\{f_{ij}\colon C_{ij}\to B_i\}_{i\in I_j}$ for each $j\in J$, the composition\[
    \{g_j\circ f_{ij}\colon C_{ij}\to A\}_{j\in J, i\in I_j}
    \] is also a covering family.
\end{itemize}
A morphism of categories with covering families is a functor which preserves covers and distinguished objects. Categories with covering families and functors between them assemble into a category denoted $\bcat{CatFam}$.
\end{definition}

Categories with covering families generalize the assemblers of Zakharevich \cite{zakharevich:16a}, in that every assembler gives rise to a category with covering families with the same $K$-theory (see \cite[Example 2.8 and Remark 2.18]{bohmann/gerhardt/malkiewich/merling/zakharevich:23}). An assembler can be thought of as a category with covering families that has extra structure.
The $K$-theory construction of a category with covering families is identical to the one for assemblers, both relying on the following definition of a category of covers.

\begin{definition}\label{defn:W(C)}
Let $\cat C\in \bcat{CatFam}$ and define its \textit{category of covers} to be the category $W(\cat C)$ whose objects are any finite collection $\{A_i\}_{i\in I}$ of non-distinguished objects of $\cat C$ (where $I$ is a finite set). The natural basepoint object of $W(\cat C)$ is the empty family $\varnothing$. A morphism $f\colon \{A_i\}_{i\in I}\to \{B_j\}_{j\in J}$ in $W(\cat C)$ is a set-map $f\colon I\to J$ along with morphisms $f_i\colon A_i\to B_{f(i)}$ for all $i\in I$ so that\[
\{f_i\colon A_i\to B_j\}_{i\in f^{-1}(j)}
\] is a covering family for each $j\in J$. This construction defines a functor $W\colon \bcat{CatFam}\to \bcat{Cat}$.
\end{definition}

Using this category, the authors of \cite{bohmann/gerhardt/malkiewich/merling/zakharevich:23} construct the \textit{$K$-theory spectrum} $K(\cat C)$ of a category with covering families $\cat C$ as a symmetric spectrum. We will not go into the full details, instead pointing the reader to \cite[Definition 2.17]{bohmann/gerhardt/malkiewich/merling/zakharevich:23}, however we will highlight a few observations about the $K$-groups of a category with covering families which will be useful for us.

\begin{definition}
The $n^{th}$ $K$-group of $\cat C$ is $
K_n(\cat C) := \pi_nK(\cat C).
$
\end{definition}

These $K$-groups are very difficult to compute in general, unless $n=0$, in which case there is a concrete description given in \cite[Proposition 3.8]{bohmann/gerhardt/malkiewich/merling/zakharevich:23} (or see \cite[Theorem 2.13]{zakharevich:16a} for the analogous result for assemblers).

\begin{theorem}[$K_0$ theorem]\label{thm:K0}
For a category with covering families $\cat C$, \[
K_0(\cat C) =  \ZZ[\ob\cat C] \Big/[A]=\sum_{i\in I}[A_i]
\] for any covering family $\{A_i\to A\}_{i\in I}$.
\end{theorem}

\begin{remark}\label{rmk:UP of K0}
    The group $K_0(\cat C)$ is the universal home for \textit{covering invariants}, which are assignments $F\colon \ob\cat C\to \cat A$ valued in some Abelian group $\cat A$ satisfying $F(A) = \sum_{i\in I} F(A_i)$ for every cover $\{A_i\to A\}_{i\in I}$. It follows from the theorem above that any covering invariant factors uniquely as\[
\begin{tikzcd}
    \ob \cat C \ar[r, "F"] \ar[d] & \cat A\\
    K_0(\cat C) \ar[ur, dashed, swap, "{F'}"] &
\end{tikzcd}
\] where $\ob\cat C\to K_0(\cat C)$ sends an object to its equivalence class.
\end{remark}

In higher algebraic $K$-theory, there are many fundamental theorems that make $K$-theory easier to compute (see, for example, \cite[\S V]{weibel:13}). 
With small adjustments, Zakharevich's proofs of these theorems for the $K$-theory of assemblers \cite[Theorems A,B,C,D]{zakharevich:16a} can be adapted to the setting of categories with covering families. We provide references to the relevant results of Zakharevich for completeness.

\begin{theorem}\label{thm: equiv of cats on KT}
    Suppose $F\colon \cat C\to \cat D$ is a morphism between categories with covering families which is an underlying equivalence of categories. Then \[KF\colon K(\cat C)\to K(\cat D)\] is a weak equivalence of spectra.
\end{theorem}
\begin{proof}
    This is the analog of \cite[Corollary 4.2]{zakharevich:16a} and \cite[Lemma 4.1]{zakharevich:16a}, which relies on \cite[Proposition 2.11(3)]{zakharevich:16a}. The proofs for all of these results go through verbatim for categories with covering families, since none of them rely on the extra assembler structure.
\end{proof}

The D\'evissage Theorem provides sufficient conditions for the inclusion of a full subcategory with covering families $\cat D\hookrightarrow\cat C$ to induce an equivalence on $K$-theory. By full subcategory with covering families, we mean that $\cat D$ is a full subcategory of $\cat C$ and that $\cat D\to \cat C$ is a morphism of categories with covering families. In particular, every covering family in $\cat D$ is sent to a covering family in $\cat C$; this implies that $W(\cat D)$ is a full subcategory of $W(\cat C)$. The proof of the D\'evissage Theorem \cite[Theorem B]{zakharevich:16a} for assemblers also readily adapts to categories with covering families, but we need to impose the extra condition that covers have common refinements. Recall that a cover $\{f_i : A_i \rightarrow A \}_{i \in I}$ is a \textit{refinement} of a cover $\{g_j : B_j \rightarrow A \}$ if for every $i \in I$ there is a $j \in J$ and a map $h:A_i \rightarrow B_j$ such that $f_i = g_j \circ h$. 

\begin{definition}
    A category with covering families \textit{has refinements} if any pair of covers $\{A_i\to A\}_{i\in I}$ and $\{A'_j\to A\}_{j\in J}$ of the same object $A$ has a common refinement.
\end{definition}

\begin{theorem}[D\'evissage]\label{thm:devissage}
Suppose $\cat C$ has refinements and $i\colon \cat D\to \cat C$ is an inclusion of a full subcategory with covering families. If every object $C\in \cat C$ has a cover $\{D_i\to C\}_{i\in I}$ with $D_i\in \cat D$ for all $i\in I$, then \[
Ki\colon K(\cat D)\xrightarrow{\sim} K(\cat C)
\] is an equivalence.
\end{theorem}\begin{proof}
    If $\cat C$ has refinements, then the analogy of \cite[Proposition 2.11(2)]{zakharevich:16a} holds, which says that any diagram $A\to B\leftarrow C$ in $W(\cat C)$ can be completed to a commutative square. The proof of \cite[Theorem B]{zakharevich:16a} then works verbatim.
\end{proof}

Another helpful theorem, the Localization Theorem, is stated for simplicial categories with covering families.

\begin{definition}
    A \textit{simplicial category with covering families} is a simplicial object in $\bcat{CatFam}$, i.e. a functor $\Delta^{\op}\to \bcat{CatFam}$. The $K$-theory of a simplicial category with covering families is the homotopy colimit
    \[
    K(\cat C_\bullet) \simeq \hocolim_{[n]\in\Delta^{\op}} K(\cat C_n).
    \] 
\end{definition}

In particular, if $\cat C_\bullet$ is the constant simplicial object on a category with covering families $\cat C$, then $K(\cat C_\bullet)\simeq K(\cat C)$. 

\begin{theorem}[Localization]\label{thm:localization}
    Let $F\colon \cat D_\bullet\to \cat C_\bullet$ be a morphism of simplicial categories with covering families. Then there is a simplicial category with covering families $\cat C/F_\bullet$ so that\[
    K(\cat D_\bullet)\xrightarrow{KF} K(\cat C_\bullet)\to K(\cat C/F_\bullet)
    \] is a cofiber sequence.
\end{theorem}\begin{proof}
    The simplicial category with covering families $C/F_\bullet$ is analogous to the simplicial assembler described in \cite[Definition 6.1]{zakharevich:16a}. The proof of \cite[Theorem C]{zakharevich:16a} then works verbatim, with most of the work being done in \cite[Lemma 6.2]{zakharevich:16a} (which again makes use of \cite[Proposition 3.11(3)]{zakharevich:16a}).
\end{proof}

\begin{remark}
    As in \cite[Corollary 6.3]{zakharevich:16a}, the Localization Theorem applied to the collapse map $\cat C\to *$ implies that $K(\cat C)$ is a $\Omega$-spectrum.
\end{remark}

\begin{remark}
We will describe $\cat C/F_\bullet$ in the case when $F\colon \cat D\to \cat C$ is a functor between categories with covering families (considered as constant simplicial categories), pointing the reader to \cite[Definition 6.1]{zakharevich:16a} for the more general construction. 

Given two categories with covering families, their coproduct $\cat C\vee \cat D$ is another category with covering families. This category has objects $\ob\cat C\vee \ob\cat D$ where we identify the distinguished object of $\cat C$ with the distinguished object of $\cat D$. The morphisms between non-initial objects are given by $\hom\cat C\amalg \hom \cat D$.
More generally, given any set $I$ and $I$-tuple of categories with covering families $\{\cat C_i\}_{i\in I}$, we get a category with covering families $\bigvee_{i\in I} \cat C_i$.

The simplicial category with covering families $(\cat C/F)_\bullet$ is specified by\[
(\cat C/F)_n = \cat C \lor \left(\bigvee_{i=1}^n \cat D\right)
\] and the face and degeneracy maps are given as follows:\begin{align*}
    d_i\colon \cat C\vee \left(\bigvee_{i=1}^n \cat D\right)&\xrightarrow{1\vee \nabla_i}\cat C\vee \left(\bigvee_{i=1}^{n-1} \cat D\right) & 0<i \leq n,\\
    d_0\colon \cat C\vee \left(\bigvee_{i=1}^n \cat D\right) &\xrightarrow{\cong} \cat C\vee \cat D\vee \left(\bigvee_{i=1}^{n-1} \cat D\right)\xrightarrow{\nabla(F)\vee 1} \cat C\vee \left(\bigvee_{i=1}^{n-1} \cat D\right) & i=0,\\
    s_j\colon \cat C\vee \left(\bigvee_{i=1}^n \cat D\right) &\xrightarrow{1\vee s^j} \cat C\vee \left(\bigvee_{i=1}^{n+1} \cat D\right) & 0\leq j\leq n,
\end{align*} 
where $\nabla_i$ applies the fold map $\nabla\colon \cat D\vee \cat D\to \cat D$ to the $(i-1)^{th}$ and $i^{th}$ components of $\bigvee_{i=1}^n \cat D$, $\nabla(F)$ is the restriction of the fold map $\nabla\colon \cat  C\vee \cat C\to \cat C$ to $\cat C\vee F(\cat D)$, and $s^j$ includes $\bigvee_{i=1}^n \cat D$ into $\bigvee_{i=1}^n \cat D$ by skipping the $j^{th}$ factor.
\end{remark}

\begin{remark}\label{rmk:thm D analogue}
When $i\colon \cat D\to \cat C$ is the inclusion of a full subassembler satisfying certain assumptions, \cite[Theorem D]{zakharevich:16a} identifies the simplicial assembler $\cat C/i_\bullet$ with another assembler $\cat C\setminus \cat D$. To get an analogous theorem for category with covering families, we could impose similar conditions on an inclusion $i\colon \cat D\to \cat C$:
\begin{itemize}
    \item The category with covering families $\cat C$ has refinements, pullbacks, and every morphism appearing in a cover is a monomorphism.
    \item The subcategory $\cat D$ is \textit{closed under covers} in $\cat C$; meaning whenever $\{C_i\to D\}_{i\in I}$ is a cover in $\cat C$ of an object $D\in \cat D$, then it is a cover in $\cat D$ as well.
    \item The subcategory $\cat D$ \textit{has complements in $\cat C$}; meaning for all objects $D\in \ob\cat D$ and any morphism $f\colon D\to C$ in $\cat C$, there is a cover of $C$ in $\cat C$ containing $f$.
\end{itemize}
Under these assumptions, $\cat C$ is basically an assembler (the main difference being that we do not need to check that the collection of covering families makes $\cat C$ a Grothendieck site) and Zakharevich's proofs work in this context as well.
\end{remark}

As mentioned in the introduction (and discussed further in \cref{sec:abstr recon}), a reconstruction problem asks when an object $X$ can be uniquely reconstructed (up to isomorphism) from some collection of partial data $\{X_i\}_{i\in I}$. If we can find a category with covering families $\cat C$ wherein the partial information $\{X_i\}_{i\in I}$ is constructed as a multimorphism $\{X_i\to X\}_{i\in I}$, then $[X]=\sum_{i\in I}[X_i]$ in $K_0(\cat C)$ by \cref{thm:K0}. Asking whether $X$ is reconstructible up to isomorphism in $\cat C$ from the collection $\{X_i\}_{i\in I}$ is then the same as asking whether the map\[
iso(\ob\cat C)\to K_0(\cat C)
\] which sends $[X]$ to itself is an injection, assuming that isomorphisms are always covers in $\cat C$. That is, the fiber of this map encodes obstructions to reconstructibility.

It would be helpful to realize the map above as part of a long exact sequence of $K$-theory groups, but $iso(\ob \cat C)$ cannot be constructed as a $K$-group in general.
However, if we extend the map linearly to the free Abelian group \[
\ZZ[iso(\ob \cat C)]\to K_0(\cat C),
\] then we can realize $\ZZ[iso(\ob \cat C)]\cong K_0(\cat C, iso)$ where $(\cat C, iso)$ is the category $\cat C$ with only isomorphisms as covers. In fact, we can compute the entire $K$-theory spectrum of $(\cat C, iso)$ using methods from \cite[\S 5.1]{zakharevich:16a}.

\begin{proposition}\label{ex:covs are isos}
There is an equivalence of spectra \[
K(\cat C, iso) \simeq \bigvee_{[c]} \Sigma^\infty_+ B(iso_{\cat C}(c))
\] where the wedge ranges across isomorphism classes of objects in $\cat C$ and $iso_{\cat C}(c)$ is the group of $\cat C$-isomorphisms of $c$.
\end{proposition}\begin{proof}
    Let $\tilde{\cat C}=\bigvee_{[c]} \cat C_c$, where $\cat C_c$ is the full subcategory of $\cat C$ on objects with a morphism to $c$. Then $\tilde{\cat C}$ is a full subcategory with covering families in $(\cat C, iso)$ and satisfies the conditions of the D\'evissage Theorem. Hence $K(\cat C, iso)\simeq K(\tilde{\cat C})$, and by \cite[Lemma 3.6]{bohmann/gerhardt/malkiewich/merling/zakharevich:23}, $K(\tilde{\cat C})\simeq \bigvee_{c} K(\cat C_c)$. The fact that $K(\cat C_c)\simeq \Sigma^\infty_+ B(iso_{\cat C}(c))$ follows from \cite[Example 4.6]{bohmann/gerhardt/malkiewich/merling/zakharevich:23}.
\end{proof}

\section{Reframing the edge reconstruction conjecture}\label{sec:graphs}
Recall from the introduction that the edge reconstruction conjecture asks whether a graph is determined (up to isomorphism) by a collection of partial information known as the edge deck. In order to reframe this conjecture in the language of $K$-theory, we will need some more precise definitions.

\begin{definition}
The category \textbf{FinGraph} consists of\begin{itemize}
    \item Objects: A finite graph is a pair of finite sets $(V,E)$ equipped with an injective function $\partial : E \rightarrow \text{Pair}(V)$, where $\text{Pair}(V)$ denotes the set of unordered pairs of distinct vertices from $V$. Two vertices which are connected by an edge are said to be \textit{adjacent}.
    \item Morphisms: a graph morphism $f:(V,E, \partial) \rightarrow (V',E', \partial')$ is a pair of injective set maps $f_V : V \rightarrow V'$ and $f_E:E \rightarrow E'$ such that for all $e \in E$, we have $\partial' f_E(e) = f_V ( \partial e)$.
\end{itemize}
\end{definition}
The graphs in \textbf{FinGraph} are finite, irreflexive, undirected, unweighted, and do not have multiple edges between pairs of vertices. The morphisms are inclusions of graphs. 

Given a graph $G$ with vertex set $V$ and edge set $E$, we write $G-e$ for the \textit{edge-deleted subgraph} of $G$ for $e\in E$, which has vertex set $V$ and edge set $E\setminus\{e\}$. 

\begin{definition}
The \textit{edge deck} of a graph $G$ is the multiset of all edge-deleted subgraphs of $G$, up to isomorphism. This multiset is denoted $\ED(G)$. An element of $\ED(G)$ is called an \textit{edge-card}, or just \textit{card} for simplicity.
\end{definition}

\begin{definition}
A graph $G'$ is called an \textit{edge-reconstruction} of $G$ if $\ED(G')=\ED(G)$. A graph $G$ is called \textit{edge-reconstructible} if any edge-reconstruction $G'$ is isomorphic to $G$. 
\end{definition}

The edge reconstruction conjecture is that every graph is uniquely reconstructable from its edge-deck when the number of edges is $n\geq 4$. 

\begin{conjecture}[Edge Reconstruction]\label{ERC}
    Let $G$ and $G'$ be $n$-edged graphs, for $n\geq 4$. If $\ED(G)=\ED(G')$, then $G\cong G'$. 
\end{conjecture}

Note that we need the condition $n\geq 4$ since the following graphs provide counterexamples for $n=2,3$: 
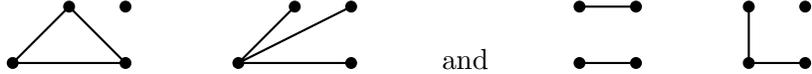
\[
\begin{tikzpicture}[scale=0.75]
    \fill (1,0) circle (3pt);
    \fill (-1,0) circle (3pt);
    \fill (0,1) circle (3pt);
    \fill (1,1) circle (3pt);
    \draw[thick] (0,1) -- (1,0) -- (-1,0) -- cycle;

    \fill (5,0) circle (3pt);
    \fill (3,0) circle (3pt);
    \fill (4,1) circle (3pt);
    \fill (5,1) circle (3pt);
    \draw[thick] (5,0) -- (3,0);
    \draw[thick] (5,1) -- (3,0) -- (4,1);
\end{tikzpicture}
\hspace{1cm}\text{ and } \hspace{1cm}
\begin{tikzpicture}[scale=0.75]
    \fill (1,0) circle (3pt);
    \fill (0,0) circle (3pt);
    \fill (1,1) circle (3pt);
    \fill (0,1) circle (3pt);
    
    \draw[thick] (1,0) -- (0,0);
    \draw[thick] (0,1) -- (1,1);
    
    \fill (4,0) circle (3pt);
    \fill (3,0) circle (3pt);
    \fill (4,1) circle (3pt);
    \fill (3,1) circle (3pt);
    
    \draw[thick] (4,0) -- (3,0);
    \draw[thick] (3,0) -- (3,1);
\end{tikzpicture}.
\]

Edge reconstruction can also be studied from the perspective of \textit{(edge) reconstruction invariants}, which are those graph invariants that take the same value on reconstructions.

\begin{definition}
    A graph invariant $f\colon \ob(\textbf{FinGraph})\to A$ valued in an Abelian group $A$ is an \textit{edge reconstruction invariant} if $f(G) = f(G')$ whenever $\ED(G)= \ED(G')$.
\end{definition}

There are many graph invariants that are known to be reconstruction invariants (for $n\geq 4$), such as the number of isolated vertices and Euler characteristic (since both $\abs{V(G)}$ and $\abs{E(G)}$ may be deduced from $\ED(G)$). Moreover, many graph polynomials are vertex reconstruction invariants such as the Tutte polynomial, characteristic polynomial and the chromatic polynomial, among others \cite{kotek} (see also \cite{forman}). 

\subsection{Covering families on graphs}\label{subsec:cov fam on FG}
For $n \in \mathbb{N}$, let $\Gamma_{\leq n}$ denote the full subcategory of $\bcat{FinGraph}$ spanned by graphs with no more than $n$ edges. We get a filtration of $\bcat{FinGraph}$ as
$$\Gamma_{\leq 0} \into \Gamma_{\leq 1} \into \cdots \into \Gamma_{\leq n} \into \cdots$$

When imagining how a graph $G$ might be ``covered'' by subgraphs, an intuitive idea is to ask for a collection of subgraphs whose union is all of $G$ and whose pairwise intersection is as small as possible (i.e. only vertices). We can make this idea precise via a particular category with covering families structure on $\Gamma_{\leq n}$.

\begin{definition}
    Define a collection of covering families on $\Gamma_{\leq n}$ as follows: \begin{itemize}
    \item for every finite (possibly empty) set $I$, $\{\varnothing=\varnothing\}_{i\in I}$ is a covering family;
    \item a collection $\{f_i\colon G_i\to G\}_{i\in I}$ is a covering family if each $f_i$ is the inclusion of $G_i$ as a subgraph of $G$ so that $G=\cup_i G_i$, and for each $i\neq j$ the intersection $G_i\cap G_j$ in $G$ is a (possibly empty) set of vertices of $G$ (i.e. $G_i$ and $G_j$ have no overlapping edges as subgraphs of $G$). 
    \end{itemize}
\end{definition}
It is straightforward to check that $\Gamma_{\leq n}$ is a category with covering families, and also that $\Gamma_{\leq n}$ has refinements. Although this category with covering families structure on $\Gamma_{\leq n}$ is not related to the edge reconstruction conjecture, we think it is both natural and helpful to study.

Note that in this category with covering families, $K$-theory will not care about isolated vertices; that is, $\{ * = *\}_{i\in I}$ is a cover of $*$ for any $I\neq \varnothing$. In particular, this implies that $[*]=0$ in $K_0(\Gamma_{\leq n})$ and consequently $[G] = [G'\amalg (\amalg_k *)] = [G']$ where $G'$ is the full subgraph of $G$ on non-isolated vertices. We will use this observation, along with D\'evissage, to compute the entire $K$-theory spectrum of $\Gamma_{\leq n}$.

\begin{proposition}\label{RPinf KGam1}
    For all $n\geq 1$, there is an equivalence of $K$-theory spectra\[
    \Sigma^\infty_+ \RR\PP^\infty \simeq K(\Gamma_{\leq 1}) \xrightarrow{\sim} K(\Gamma_{\leq n})
    \] induced by the inclusion $\Gamma_{\leq 1}\to \Gamma_{\leq n}$.
\end{proposition}\begin{proof}
    Every graph $G$ with $\leq n$ edges admits a cover by objects of $\Gamma_{\leq 1}$, namely by each edge of $G$ and the isolated vertices of $G$. By \cref{thm:devissage}, the inclusion $\Gamma_{\leq 1}\to \Gamma_{\leq n}$ induces an equivalence $K(\Gamma_{\leq 1})\to K(\Gamma_{\leq n})$. 
    
    To see that $K(\Gamma_{\leq 1})\simeq \Sigma^\infty_+\RR\PP^\infty$, observe that $\Gamma_{\leq 1}\cong (\{\grI\}, iso) \times \bcat{FinSet}$, where $\bcat{FinSet}$ is given the structure of a category with covering families where a cover of $\underline n$ is any collection of injections $\{\underline{n}_i\to \underline n\}_{i \in I}$ which are jointly surjective (the images may have non-empty intersection). We claim that with this structure, the $K$-theory of $\bcat{FinSet}$ is contractible, and consequently $K(\Gamma_{\leq 1})\simeq K(\{\grI\}, iso)\simeq \Sigma^\infty_+ B(\ZZ/2)\simeq \Sigma^\infty_+ \RR\PP^\infty$ by \cref{ex:covs are isos}. We emphasize that this category with covering structure on $\bcat{FinSet}$ is an unusual one, and so this claim does not contradict the Barratt--Priddy--Quillen Theorem.

    To see that $K(\bcat{FinSet})\simeq *$ in this case, we study its category of covers $W(\bcat{FinSet})$. The proof is similar to the Eilenberg swindle (see the \textit{flasque categories} of \cite[V.1.9]{weibel:13}). Consider the duplication functor $\id\vee \id\colon W(\bcat{FinSet})\to W(\bcat{FinSet})$ that sends a tuple $\{\underline{n}_i\}_{i\in I}$ to two copies of that tuple, $\{\underline{n}_i\}_{i\in I}\amalg \{\underline{n}_i\}_{i\in I}$. There is a natural transformation $\nabla$ from $\id\vee \id$ to the identity, since for every object $\underline{n}_i$, the multimorphism $\{\id_{\underline{n}}, \id_{\underline{n}}\}$ consisting of two copies of the identity is a cover. Then $B(\nabla)$ is a homotopy between the identity and $B(\id \vee \id)\simeq B(\id) \vee B(\id)$. This means that the identity on $W(\bcat{FinSet})$ is null-homotopic, which implies $K(\bcat{FinSet})\simeq *$ as desired.
\end{proof}

In particular, taking the colimit, we can compute $K(\bcat{FinGraph})$ for this natural notion of cover. 

\begin{corollary}\label{fingraph spectrum}
   There is an equivalence of spectra $K(\bcat{FinGraph})\simeq \Sigma^\infty_+\RR\PP^\infty$.
\end{corollary}

It is perhaps notable and surprising that we can compute the entire $K$-theory spectrum for this evident notion of cover, as $K$-theory computations are notoriously difficult. We expect that there are other notions of covers suitable to finite graphs that may yield more interesting $K$-theories. The next one we introduce is constructed specifically for the purposes of studying edge reconstruction.

Let $\Gamma_n$ be the full subcategory on $\Gamma_{\leq n} \setminus \Gamma_{\leq n-1}$ and the initial object $\varnothing$, i.e. graphs with exactly $n$ edges and isomorphisms between them (along with the initial object and morphisms).

\begin{definition}
Define a collection of covering families on $\Gamma_{n}$ as follows: \begin{itemize}
    \item for every finite (possibly empty) set $I$, $\{\varnothing=\varnothing\}_{i\in I}$ is a covering family;
    \item for every isomorphism $f\colon G\xrightarrow{\cong} G'$ in $\Gamma_n$, the singleton $\{f\colon G\xrightarrow{\cong} G'\}$ is a covering family.
\end{itemize}
Then it is straightforward to check that $\Gamma_n$ is a category with covering families.
\end{definition}

\begin{remark}\label{rmk:K0 of Gamma k}
    By \cref{thm:K0}, $K_0(\Gamma_n) \cong \ZZ[\mathcal G_n]$, where $\mathcal G_n$ is the set of isomorphism classes of graphs with $n$ edges. In fact, by \cref{ex:covs are isos}, we can compute\[
    K(\Gamma_n) \simeq \bigvee_{[G]} \Sigma^\infty_+ B(iso_{\bcat{FinGraph}}(G))
    \] where the wedge ranges over isomorphism classes of graphs with $n$ edges.
\end{remark}

Recall that the \textit{edge deck} of a graph $G$ is the multiset of all edge-deleted subgraphs of $G$, up to isomorphism. This multiset is denoted $\ED(G)$, and an element of $\ED(G)$ is called a \textit{edge-card} or just \textit{card} for simplicity. The edge reconstruction conjecture is that every graph is uniquely reconstructable from its edge deck when the number of edges is $n\geq 4$. In other words, if $\ED(G)=\ED(G')$, then $G\cong G'$.

We will now construct a version of $\Gamma_n$ which encodes edge decks. Let $\Gamma_{n,n-1}$ be the full subcategory on $\Gamma_{\leq n}\setminus \Gamma_{\leq n-2}$ and the initial object $\varnothing$; that is, $\Gamma_{n,n-1}$ consists of graphs with $n$ or $n-1$ edges and subgraph inclusions between them (along with the initial object and morphisms).

\begin{definition}\label{defn:gamma for ED}
    Define a collection of covering families on $\Gamma_{n,n-1}$ as follows: \begin{itemize}
    \item for every finite (possibly empty) set $I$, $\{\varnothing=\varnothing\}_{i\in I}$ is a covering family;
    \item for every isomorphism $f\colon G\xrightarrow{\cong} G'$ in $\Gamma_{n,n-1}$, the singleton \[\{f\colon G\xrightarrow{\cong} G'\}\] is a covering family;
    \item if $G$ is a graph with $n$ edges, enumerate $\ED(G)=\{C_1,\dots,C_n\}$ and fix the inclusions of the cards $i_j\colon C_j\to G$. Then \[
    \{f_j\colon C'_j\to G\}_{j\in[n]} 
    \] is a covering family if and only if there are isomorphisms $g_j\colon C'_j\to C_j$ so that $f_j=i_j\circ g_j$ for all $j=1,\dots,n$. 
\end{itemize}
Then $\Gamma_n$ is a category with covering families.
\end{definition}

\begin{remark}
    In the language of \cref{sec:abstr recon}, we may view $\Gamma_{n,n-1}$ as a reconstruction setting with distinguished object $\varnothing$, small objects $(n-1)$-edged graphs, large objects $n$-edge graphs, and decompositions given by edge decks. The induced categories with covering families structure given by \cref{ex:rec setting covering fams 2} is the same as the one given by \cref{defn:gamma for ED}. Essentially, the only cover in $\Gamma_{n,n-1}$ that is not given by an isomorphism is the edge deck of a graph with $n$-edges.
\end{remark}

\begin{remark}
By \cref{thm:K0}, we know \[  K_0(\Gamma_{n,n-1}) \cong \ZZ[\ob\Gamma_{n,n-1}]/\sim.
\] where $\sim$ is generated by $[G]=[G']$ whenever $G\cong G'$ as well as $[G] = \sum_{C\in \ED(G)} [C]$ when $G$ has $n$ edges. Note that inclusion functor $\Gamma_n\to \Gamma_{n,n-1}$ is a morphism of categories with covering families, which gives us a well-defined map on $K_0$. In particular, by \cref{rmk:K0 of Gamma k}, we get a well-defined map \[
\iota_n\colon \mathcal G_n \to K_0(\Gamma_{n,n-1})
\] which sends $[G]$ to itself.
\end{remark}

\begin{proposition}\label{prop:equal ED iff equal in K0}
    Let $G$ and $G'$ be graphs with $n$ edges. Then $\ED(G)=\ED(G')$ if and only if $[G]=[G']$ in $K_0(\Gamma_{n,n-1})$.
\end{proposition}
\begin{proof}
     The forward direction is true by construction, and for the other direction note that if $[G]=[G']$ in $K_0(\Gamma_{n,n-1})$ then $\sum_{C\in \ED(G)}[C]=\sum_{C'\in \ED(G')} [C']$. Then there is a bijective correspondence between $\ED(G)$ and $\ED(G')$ so that $[C]=[C']$ in $K_0(\Gamma_{n,n-1})$, but for two graphs with $n-1$ edges to be equal in $K_0(\Gamma_{n,n-1})$, they must be isomorphic. Thus $C\cong C'$ and so $\ED(G)=\ED(G')$. 
\end{proof}

The description of $K_0(\Gamma_{n,n-1})$ above implies that $K_0(\Gamma_{n,n-1})$ is the universal home for reconstruction invariants, in the following sense.

\begin{corollary}\label{cor:univ prop reconstr invar}
     Let $f\colon \mathcal {G}_n\to A$ be a reconstruction invariant valued in an Abelian group $A$. Then $f$ factors uniquely as\[
    \begin{tikzcd}
        \mathcal{G}_n \ar[rd, swap, "f"] \ar[r, "i_n"] & K_0(\Gamma_{n,n-1}) \ar[d, dashed, "\hat f"]\\
        & A
    \end{tikzcd}
    \]
\end{corollary}

We can now reformulate the edge reconstruction conjecture as a statement about $K_0(\Gamma_{n,n-1})$.

\begin{corollary}\label{cor:VRC iff inj}
    The edge reconstruction conjecture is equivalent to the statement that the map\[
   i_n\colon \mathcal G_n \to K_0(\Gamma_{n,n-1})
    \] which sends $[G]$ to itself is an injection for all $n\geq 4$.
\end{corollary}
\begin{proof}
    By \cref{prop:equal ED iff equal in K0}, the statement that $\ED(G)=\ED(G')$ is equivalent to the statement that $[G]=[G']$ in $K_0(\Gamma_{n,n-1})$. Therefore the edge reconstruction conjecture is equivalent to the statement that if $[G]=[G']$ in $K_0(\Gamma_{n,n-1})$, then $G\cong G'$, which is exactly to say that $\iota$ is injective. 
\end{proof}

\begin{remark}
    A more general version of the edge reconstruction conjecture considers $k$-edge decks, $\ED_k(G)$, which is the multiset of all subgraphs formed by removing $k$ edges from $G$. The $k$-edge reconstruction conjecture states that a graph with $n\geq k+3$ edges is uniquely reconstructable from $\ED_k(G)$. Our work can be generalized to this setting by replacing $n-1$ with $n-k$ everywhere.
\end{remark}

The inclusions of categories with covering families $\Gamma_n\hookrightarrow \Gamma_{n,n-1}$ and $\Gamma_{n-1}\hookrightarrow \Gamma_{n,n-1}$ have interesting consequences on $K$-theory. We state the implications of the latter map here, but reserve our discussion of the former map for the next subsection.

\begin{theorem}\label{thm:dev on Gamma}
    The inclusion $\Gamma_{n-1}\hookrightarrow \Gamma_{n,n-1}$ induces an equivalence on $K$-theory,\[
    K(\Gamma_{n-1}) \xrightarrow{\simeq} K(\Gamma_{n,n-1}).
    \]
\end{theorem}
\begin{proof}
    This is a simple application of D\'evissage (\cref{thm:devissage}). The refinement condition on $\Gamma_{n,n-1}$ is straightfoward to check.
\end{proof}

In particular, by \cref{ex:covs are isos} (c.f. \cref{rmk:K0 of Gamma k}), we can identify the entire $K$-theory spectrum of $\Gamma_{n,n-1}$.

\begin{corollary}
    There is an equivalence of spectra\[
    K(\Gamma_{n,n-1}) \simeq \bigvee_{[G]} \Sigma^\infty_+ B{\rm Aut}(G)
    \] where the wedge ranges over isomorphism classes of graphs $G$ with $n-1$ edges and ${\rm Aut}(G)$ is the group of graph isomorphisms of $G$.
\end{corollary}

The inclusion of categories with covering families $\iota_n\colon \Gamma_n\into \Gamma_{n,n-1}$ induces a map on $K_0$ which captures the edge reconstruction conjecture in the sense that if the map $K_0(\iota_n)\colon K_0(\Gamma_n)\to K_0(\Gamma_{n,n-1})$ is injective then the edge reconstruction conjecture is true for graphs with $n$ edges. To see this, note that for any $G,G'\in \Gamma_n$, we have $[G]=[G']$ in $K_0(\Gamma_{n,n-1})$ if and only if $\ED(G)=\ED(G')$. If $K_0(\iota_n)$ is injective, then this implies $[G]=[G']$ in $K_0(\Gamma_n)$, i.e. $G\cong G'$ in $\Gamma_n$. One might hope to use the Localization Theorem (\cref{thm:localization}), which tells us that\[
K_1((\Gamma_{n,n-1}/\iota_n)_\bullet)\xrightarrow{\bndry} K_0(\Gamma_n)\xrightarrow{K_0(\iota_n)} K_0(\Gamma_{n,n-1}),
\] is exact, so it would suffice to show that the image of $\bndry$ is trivial. However, as we show in the next subsection, the map $K_0(\iota_n)$ is not injective. This does not disprove the edge reconstruction conjecture, but rather a stronger version of it, as discussed in the next subsection. 

Nonetheless, it is possible that a better understanding of $K_1((\Gamma_{n,n-1}/\iota_n)_\bullet)$ would shed light on the kernel of $K_0(\iota_n)$ and the edge reconstruction conjecture. The image of $\bndry$ is closely related to the edge reconstruction conjecture: a counterexample to the edge reconstruction conjecture is precisely a non-trivial element of the image of $\bndry$ of the form $[G]-[G']$ for graphs with $n$-edges $G,G'$. One could use methods similar to \cite{zakharevich:16b} to give a description of $K_1$ of a category with covering families and approach the problem by studying $K_1((\Gamma_{n,n-1}/\iota_n)_\bullet)$ and $\partial$.
\subsection{Generalized Edge Reconstruction}\label{subsec:general ERC}

Although $K_0(\iota_n)$ being injective implies the edge reconstruction conjecture is true for graphs with $n$ edges, the converse direction does not hold. If the edge reconstruction conjecture is true, then $\ED(G)=\ED(G')$ implies $G\cong G'$ and therefore $[G]=[G']$ in $K_0(\Gamma_n)$, but this is not enough to conclude $K_0(\iota_n)$ is injective, since the kernel may contain formal sums. However, we can use this observation to rephrase injectivity of $K_0(\iota_n)$ as equivalent to a much stronger version of the edge reconstruction conjecture.

\begin{conjecture}[Generalized edge reconstruction conjecture]
    Fix $n\geq 4$. For any $k\geq 1$, a collection $\{G_1,\dots, G_k\}$ of $n$-edged graphs is uniquely reconstructible from $\coprod_{i=1}^k \ED(G_i)$.
\end{conjecture}

In other words, if $\{G_1,\dots, G_k\}$ and $\{G_1',\dots, G_k'\}$ are two collections of $n$-edged graphs so that $\coprod_{i=1}^k\ED(G_i)\cong \coprod_{i=1}^k \ED(G_i')$ as multisets, then $\{G_1,\dots, G_k\}\cong \{G_1,\dots, G_k'\}$ as multisets. This much stronger conjecture is equivalent to injectivity of $K_0(\iota_n)$.

\begin{theorem}\label{thm:inj of i iff GERC}
    The map $K_0(\iota_n)\colon K_0(\Gamma_n)\to K_0(\Gamma_{n,n-1})$ is injective if and only if the generalized edge reconstruction conjecture is true for graphs with $n$ edges.
\end{theorem}
\begin{proof}
Suppose $K_0(\iota_n)$ is injective, and consider two subsets $\{G_1,\dots, G_k\}$ and $\{G_1',\dots, G_k'\}$ such that $\coprod_{i=1}^k\ED(G_i)\cong \coprod_{i=1}^k \ED(G_i')$. Then\[
K_0(\iota_n)(\sum_{i=1}^k [G_i] - \sum_{i=1}^k [G_i']) = \sum_{\substack{i=1\\C\in \ED(G_i)}}^k [C] - \sum_{\substack{i=1\\C'\in \ED(G'_i)}}^k [C'] = 0,
    \] and hence $\sum_{i=1}^k [G_i] = \sum_{i=1}^k [G_i']$ as elements of $K_0(\Gamma_n)$.  It follows that $\{G_1,\dots, G_k\}\cong \{G_1,\dots, G_k'\}$ as multisets.

    If $K_0(\iota_n)$ is not injective, then there is some non-zero formal sum $\sum_{i=1}^k [G_i] - \sum_{i=1}^j [G'_i]$ which $K_0(i)$ sends to $0$. Then \[
    0 = K_0(\iota_n)(\sum_{i=1}^k [G_i] - \sum_{i=1}^j [G'_i]) = \sum_{\substack{i=1\\C\in \ED(G_i)}}^k [C] - \sum_{\substack{i=1\\C'\in \ED(G'_i)}}^j [C']
    \] which implies that $k=j$ and the collections $\{G_1,\dots, G_k\}$ and $\{G_1',\dots, G_k'\}$ are a counterexample to the generalized edge reconstruction conjecture.
\end{proof}

Now, using the D\'evissage Theorem (\cref{thm:devissage}), we can show that $K_0(\iota_n)$ is never injective and consequently this stronger version of the edge reconstruction conjecture must be false.

\begin{theorem}\label{thm:GERC false}
    For every $n\geq 2$, the generalized edge reconstruction conjecture is false.
\end{theorem}\begin{proof}
    The usual edge reconstruction conjecture fails for $n=2,3$ and we will prove that $K_0(\iota_n)$ fails to be injective for every $n\geq 4$.
    
    A consequence of \cref{thm:dev on Gamma} is that the map $K_0(\Gamma_{n-1})\to K_0(\Gamma_{n,n-1})$ is an isomorphism, and the inverse is the identity on $(n-1)$-graphs and sends a graph $G$ with $n$ edges to the sum $\sum_{C\in \ED(G)}[C]$. This gives us a commutative diagram:\[
\begin{tikzcd}
    K_0(\Gamma_n) \ar[r, "K_0(\iota_n)"] \ar[rd, dashed, swap, "\varepsilon_n"] & K_0(\Gamma_{n,n-1})\ar[d, "\cong"] \\
    & K_0(\Gamma_{n-1})
\end{tikzcd}
\] so that $K_0(\iota_n)$ is injective if and only if $\varepsilon_n$ is. However, by \cref{thm:K0}, \[
\varepsilon_n\colon \ZZ[\mathcal G_n]\to \ZZ[\mathcal G_{n-1}]
\] where $\mathcal G_n$ is the set of isomorphism classes of graphs with $n$ edges. Since $\abs{\mathcal G_n}>\abs{\mathcal G_{n-1}}$, this map cannot be injective.
\end{proof}

\begin{remark}
    Our work on reframing edge reconstruction in terms of $K$-theory is what led us to this disproof the generalized edge reconstruction conjecture. However, $K$-theory is certainly not needed to do so, and the same result can obtained with elementary methods.
\end{remark}

Nonetheless, we now know that for each $n$ there is always some collection of $n$-edged graphs $\{G_1,\dots, G_k\}$ which is not uniquely reconstructible from the disjoint union of the edge decks, for some $k\geq 1$.
In fact, there are infinitely many such $k$, since if it fails for some $k$ then it fails for everything $\geq k$. This observation ensures that the following definition make sense.

\begin{definition}\label{defn:kn}
    For $n\geq 2$, let $k_n$ denote the minimum $k$ for which generalized edge-reconstruction for $n$-edge graphs fails.
\end{definition}

For instance, we know that $k_2=1$ and $k_3=1$. The original edge reconstruction conjecture can be rephrased as saying that $k_n>1$ for all $n\geq 4$. This definition introduces new avenues for exploration, by trying to bound $k_n$. 

\begin{remark}
    The map $\varepsilon_n$ is given by taking the ``edge deck'' of an arbitrary element $x\in K_0(\Gamma_n)$ given by the formal sum of the edge decks. That is, if $x=\sum_{i\in I}[G_i] - \sum_{j\in J} [G_j]$, then define \[\ED(x) = \sum_{i\in I}\sum_{C\in \ED(G_i)}[C] - \sum_{j\in J} \sum_{D\in \ED(G_j)}[D] .\]
    Then, by construction, $x\in \ker \varepsilon_n$ if and only if $\ED(x)=0$. 
\end{remark}

Observe that we can give\[
K_0(\Gamma_*) = \bigoplus_{n\geq 0} K_0(\Gamma_n)
\] the structure of a graded ring, where addition is the addition of formal sums within each $K_0(\Gamma_n)$ and multiplication is given by disjoint union, $[G]\cdot [G'] = [G\amalg G']$, which is commutative. The unit is the empty graph.

\begin{proposition}
    The kernels assemble into a graded subring, \[
    \ker(\varepsilon_*) = \bigoplus_{n\geq 0} \ker(\varepsilon_n)\subseteq K_0(\Gamma_*).
    \]
\end{proposition}
\begin{proof}
It suffices to show that $\ker(\varepsilon_n)\cdot \ker(\varepsilon_m)\subseteq \ker(\varepsilon_{n+m})$. Note that \[
\ED(G\amalg G') = \sum_{C\in \ED(G)} [C\amalg G'] + \sum_{C'\in \ED(G')} [G\amalg C'],
\] which implies
\[
\iota_{n+m}([G]\cdot [G']) = \left([G']\cdot i_n([G])\right) + \left([G]\cdot i_m([G'])\right). 
\] This observation can be generalized to all elements of $K_0(\Gamma_*)$, which implies that if $x\in \ker \varepsilon_n$ and $y\in \ker \varepsilon_m$, then $x\cdot y\in \ker \varepsilon_{n+m}$.
\end{proof}

We can use this proposition to generate an upper bound on $k_n$. 

\begin{corollary}
There is an upper bound $k_{n+m}\leq 2k_n\cdot k_m$.
\end{corollary}

Since we may take $k_2=1=k_3$, we get that $k_{n+1}\leq 2k_n$ and $k_{n+2}\leq 2k_n$. In particular, this gives $k_4=k_5=k_6=2$, and more generally we get an upper bound $k_n\leq 2^{\lfloor n/3\rfloor }$. We expect that this upper bound can be vastly improved.

\begin{remark}
    We can construct explicit examples that exhibit the upper bound $k_n\leq 2^{\lfloor n/3\rfloor}$. Recall that $[\grII]-[\grL]$ is in the kernel of $\varepsilon_2$ and $[\grClaw]-[\grTri]$ is in the kernel of $\varepsilon_3$. We can use these to generate elements of $\ker \varepsilon_n$ for any other $n$. For instance, for $n=4$, \[
   \left([\grII]-[\grL]\right)^2 = [\grII \grII] + [\grL \grL] - 2[\grII \grL]
    \] is in $\ker(\varepsilon_4)$. For $n=5$, \begin{align*}
        \left([\grII]-[\grL]\right)&\cdot\left([\grClaw]-[\grTri] \right)\\
        &=[\grII \grClaw]+[\grL \grTri] -[\grL \grClaw] - [\grII \grTri]
    \end{align*}
    is in $\ker(\varepsilon_5)$.
    For $n=6$,\[
    \left([\grClaw]-[\grTri] \right)^2= [\grClaw\grClaw] +[\grTri \grTri] -2[\grClaw \grTri]
    \] is in $\ker(\varepsilon_6)$. More generally, write $n=2q+r$ for some $q\geq 1$ and $r=0,1$. If $r=0$, then consider $\left([\grII]-[\grL]\right)^q$ and if $r=1$ then consider $\left([\grII]-[\grL]\right)^{q-1}\cdot\left([\grClaw]-[\grTri]\right)$.
\end{remark}

\section{Future directions}\label{sec:future}

In this section, we summarize some possible avenues for future research. 

As mentioned in the introduction, one could try to construct spectral lifts of known reconstruction invariants, such as the Euler characteristic. While \cref{K0 factors} gives some indication of the relationship between $K_0$ and Abelian reconstruction invariants, more systematic study could shed light on the underlying algebraic structure.

While we were unable to prove the edge reconstruction conjecture through the tools of $K$-theory, as mentioned at the end of \cref{subsec:cov fam on FG}, further study of $\bndry\colon K_1((\Gamma_{n,n-1}/\iota_n)_\bullet)\to K_0(\Gamma_n)$ could determine the kernel of $K_0(\iota_n)$, and consequently resolve the edge reconstruction conjecture. In \cite{zakharevich:16a}, specific relations for elements of $K_1$ for assembler $K$-theory are given. If the $K$-theory of covers is amenable to similar techniques, perhaps the kernel could be fully determined.  

Another direction would be organized study the $k_n$ bounds from \cref{defn:kn}. While these bounds would not be able to resolve the edge reconstruction conjecture, it would be a serious advance in our understanding of the structure of edge decks in relation to multisets of graphs. 

Another potential research direction, unrelated to the reconstruction conjecture, is to explore other category with covering family structures on $\bcat{FinGraph}$ and study the resulting $K$-theory. Although \cref{fingraph spectrum} computes the $K$-theory spectrum for one very natural notion of cover, we expect that there are other notions of cover that will yield more interesting $K$-theories. It is also possible that other combinatorial $K$-theory constructions could be applied to graphs, such as the \textit{squares categories} of \cite{CKMZ:squares}, or other kinds of graphs could be studied with $K$-theory, such as digraphs \cite{CDKOSW}.

Finally, in \cref{sec:abstr recon}, we recast general reconstruction problems as the theory of essentially injective functors. While the bare-bones treatment here is sufficient to capture much of our work on edge reconstruction, we have hardly scratched the surface of general reconstruction problems. First, finding other reconstruction problems that are amenable to the tools of $K$-theory would be valuable for pushing the boundaries of applications. Second, there may be other general examples beyond atomic reconstruction settings. Third, thinking of general reconstruction problems through the lens of computation draws natural connections to computability theory. For example, the lattice structure of reconstruction degrees could be investigated in a parallel manner to the study of Turing degrees. Last, other category theoretic tools could be brought in for general reconstruction theorems. 
\appendix
\section{Abstract Reconstruction}\label{sec:abstr recon}
A reconstruction problem asks when an object $X$ can be reconstructed (up to some notion of isomorphism) from some collection of partial data $\{X_i\}_{i\in I}$. Examples of reconstruction problems can be found in a wide variety of mathematical fields, including graph theory, algebra, manifold theory, and signal processing.
\begin{itemize}
    \item The edge reconstruction conjecture: Is every graph determined (up to isomorphism) by its edge deck? \cite{harary:1964}
    \item Neukirch–Uchida-Pop theorem: Can an infinite field that is finitely generated over prime fields be reconstructed up to isomorphism from its absolute Galois group? \cite{Pop90} 
    \item Smooth diffeomorphism groups: Can a smooth manifold be reconstructed up to diffeomorphism from its group of $p$-times differentiable diffeomorphisms? \cite{Fil82}
    \item Nyquist sampling theorem: Can a continuous time signal be accurately reconstructed from discrete samples? \cite{Whi15}
\end{itemize}
In this appendix, we give an abstract framework for reconstruction problems. After laying some categorical groundwork in \cref{app sec:gen recon}, we discuss recognizable properties in \cref{app sec:recon props}. Then in \cref{app subsec:recon and CatFam}, we connect this abstract reconstruction framework to the $K$-theory of categories with covering families as discussed in \cref{sec:KT}. We end with a general class of reconstruction problems that can be viewed K-theoretically (including the edge reconstruction conjecture from \cref{sec:graphs}) in \cref{app sec: atomic}.

\subsection{General reconstruction problems}\label{app sec:gen recon}
 Recall that a functor $F\colon \cat C \rightarrow \cat D$ is \emph{essentially injective} if for all objects $C,C' \in \cat C$, if $F(C) \cong F(C')$ then $C \cong C'$.

\begin{definition}
    A \emph{reconstruction problem} is a functor $D\colon \cat C \rightarrow \cat D$ where $\cat D$ is the \emph{data category} and $D(C)$ is the \emph{data} of the object $C$. We say that $\cat C$ is \emph{$D$-reconstructable} if $D$ is essentially injective.
\end{definition}

We think of $D$ as sending an object of $\cat C$ to some data $D(C)$ in $\cat D$ which makes up the ``partial information'' used to try to reconstruct $C$. We will use the notation $C\cong_D C'$ to indicate that $D(C)\cong D(C')$ and say that $C$ and $C'$ are \textit{$D$-isomorphic}. In particular, $\cat C$ is $D$-reconstructable precisely when $C \cong_D C'$ if and only if $C\cong C'$. Let $i \colon \cat C\to \cat C_0$ be the functor that sends an object to its isomorphism class. 

\begin{definition}
    Let $D: \cat C \rightarrow \cat D$ be a reconstruction problem. A \emph{$D$-reconstruction} is a retract $r: D_0 \rightarrow \cat C_0$ of the induced map $D_0\colon \cat C_0\to \cat D_0$ on isomorphism classes.
\end{definition}

A $D$-reconstruction can be thought of as a choice of reconstruction for every data object $D(C)$. Observing that $D_0$ is injective if and only if $D$ is essentially injective, we obtain the following.

\begin{proposition}
    Let $D\colon \cat C \rightarrow \cat D$ be a reconstruction problem. Then $\cat C$ is $D$-reconstructable if and only if there is a $D$-reconstruction.
\end{proposition}

If we view $\cat D_0$ and $\cat C_0$ as discrete categories, then $\cat C$ is $D$-reconstructable if and only if $\cat C_0$ is $D_0$-reconstructable. 
In particular, if $D\colon \cat C \rightarrow \cat D$ is a fully faithful reconstruction problem, then $\cat C$ is $D$-reconstructable.
    
Initially, this framework might seem restrictive due to the requirement that the reconstruction problem must come from a functor $D$. Many reconstruction problems we are aware of, e.g.
Nyquist sampling (\cref{app ex:ny samp}), are described as set-maps.
However, it is possible to fit these seemingly non-categorical reconstruction problems into our framework, using groupoids.

\begin{example}
    Let $\cat S$ and $\cat T$ be sets equipped with equivalence relations $\sim_\cat S$, $\sim_\cat T$ (often equality). A \emph{classical reconstruction problem} is a set-function $D: \cat S \rightarrow \cat T$ such that if $a \sim_\cat S b$ in $\cat S$, then $D(a) \sim_\cat T D(b)$. The image $D(a)$ is the \emph{data of $a$}. We say that the set $\cat S$ is \emph{$D$-reconstructible} if $D(a) \sim_\cat T D(b)$ implies $ a \sim_\cat S b$.

    Now let $\tilde{\cat S}$ be the groupoid with objects $\cat S$ and a unique isomorphism $x\to y$ if $x\sim_{\cat S} y$. Similarly define $\tilde{\cat T}$. Then there is a unique functor $D': \tilde{\cat S} \rightarrow \tilde{\cat T}$ that agrees with $D$ on objects, and $D'$ is essentially injective if and only if $D'(a) \sim_\cat T D'(b) \implies a \sim_\cat S b$, as desired.
\end{example}

This translation from classical reconstruction problems to categorical reconstruction problems shows that our framework is sufficiently general. It also demonstrates that $D$-reconstructability only depends on $D^{\cong}\colon \cat C^{\cong}\to \cat D^{\cong}$, the functor induced on maximal subgroupoids. 

\begin{example}[Yoneda embedding]\label{app ex: yoneda}
    Let $\cat C$ be a small category, and $PSh(\cat C)$ denote the category of presheaves on $\cat C$.
    The Yoneda embedding $y\colon \cat C \rightarrow PSh(\cat C)$ defines a reconstruction problem, and to say that $\cat C$ is $y$-reconstructable is to say that if two objects represent the same presheaf on $\cat C$, then those objects themselves are isomorphic in $\cat C$. This is the Yoneda lemma.
\end{example}

\begin{example}[Nyquist sampling]\label{app ex:ny samp}
    Let $f\colon \RR \rightarrow \RR$ be a Lebesgue integrable function, and consider the \emph{uniform sampling} of $f$ with sample rate $S>0$
    $$\text{Samp}_S(f) := \{f(n/S) \: | \: n \in \ZZ\}.$$
    When is $f$ determined almost everywhere by its uniform sampling? 
    We may phrase this question as a reconstruction problem in the following way. Let $L^1(\RR) $ denote the set of integrable functions $f\colon \RR \to \RR$, and give $L^1(\RR)$ the structure of a groupoid by assigning a unique isomorphism between $f$ and $g$ exactly when $f=g$ almost everywhere. Uniform sampling with sample rate $S$ defines a functor
    $\text{Samp}_S \colon L^1(\RR) \rightarrow \RR^\ZZ,$
    where $\RR^\ZZ$ is viewed as a discrete category. 
    
    Then $L^1(\RR)$ is $\text{Samp}_S$-reconstructable if and only if every $f \in L^1(\RR)$ can be determined almost everywhere from its uniform sampling with sample rate $S$. The Nyquist sampling theorem \cite{Whi15} says that reconstructability holds on a smaller subgroupoid of $L^1(\RR)$ (the \textit{$B$-band limited functions}) for particular choices of $S$ (specifically $S>2B$).
\end{example}

\begin{example}[Smooth diffeomorphism groups]\label{app ex: diffeo}
    For $p \geq 1$, let $\mathbf{Sm}_p$ denote the groupoid of smooth manifolds and $C^p$-diffeomorphisms between them. There is a functor $\text{Diff}_p \colon \mathbf{Sm}_p \rightarrow \mathbf{Grp}$ that sends a manifold to its group of $C^p$-automorphisms and acts on morphisms by conjugation. Can a smooth manifold $M$ be reconstructed up to $C^p$-diffeomorphism from $\text{Diff}_p(M)$? In \cite{Fil82}, Filipkiewicz shows the answer is yes, which is to say that $\mathbf{Sm}_p$ is $\text{Diff}_p$-reconstructable.   
\end{example}

\subsection{Recognizable properties}\label{app sec:recon props}
We can form a category of reconstruction problems by considering the arrow category $\mathbf{Recon} := \vec{\mathbf{Cat}}$, where $\mathbf{Cat}$ is the category of small categories. 
An object of $\mathbf{Recon}$ is thus a functor $F\colon \cat C \rightarrow \cat D$ and a morphism is a commuting square. We may also look at a category of reconstruction problems on a specific (small) category $\cat C$. 

\begin{definition}
     The \textit{category of reconstruction problems on $\cat C$} is the slice category \[\mathbf{Recon}_{\cat C}:=~_{{1_{\cat C}}\setminus }\mathbf{Recon}.\] That is, objects of $\mathbf{Recon}_{\cat C}$ are functors with domain $\cat C$ and a morphism from $F\colon \cat C\to \cat D$ to $F'\colon \cat C\to \cat D'$ is a functor $G\colon \cat D\to \cat D'$ that makes the relevant triangle commute.
\end{definition}

For a reconstruction problem $D\colon \cat C \rightarrow \cat D$, there may be certain properties of an object $C$ that can be determined by the data $D(C)$. We call such a property a \emph{$D$-recognizable property}. 

\begin{definition}
    Let $D\colon \cat C \rightarrow \cat D$ be a reconstruction problem, and let $P\colon \cat C \rightarrow \cat E$ be a functor. Then $P$ is \emph{$D$-recognizable} if there is a map $Q_0\colon \cat D_0 \rightarrow \cat E_0$ such that $P_0 = Q_0 \circ D_0$. If $\cat C'\subseteq \cat C$ is a subcategory, then $P$ is \textit{$D$-recognizable on $\cat C'$} if $P|_{\cat C'}$ is $D|_{\cat C'}$-recognizable.
\end{definition}

For edge reconstruction of graphs, some examples to keep in mind are the number of edges, chromatic number, the set of edge-deleted subgraphs, and the set of all subgraphs.  
Another interesting example comes from binary classification (e.g. ``is this graph a tree?'') which can be encoded in the following way.

\begin{definition}
    Let $D\colon \cat C \rightarrow \cat D$ be a reconstruction problem. Let $\mathbf{bool}$ denote the discrete category with objects $\{\top, \bot\}$. A \emph{binary $D$-recognizable property} is a $D$-recognizable property $P\colon \cat C \rightarrow \mathbf{bool}$.
\end{definition}

\begin{example}
    Let $\cat C$ be a small category and consider the reconstruction problem $\cat C/(-)\colon \cat C \rightarrow \mathbf{Cat}$ that sends an object $X$ to the slice category $\cat C/X$. In general, $\cat C$ is not $\cat C/(-)$-reconstructable, e.g. if $\cat C$ is discrete. 

    Consider the property $\text{is-terminal}\colon \cat C \rightarrow \mathbf{bool}$ that sends an object $X$ to $\top$ if and only if $X$ is terminal in $\cat C$. Since terminal objects are unique up to isomorphism, $\text{is-terminal}$ is $\cat C/(-)$-recognizable, meaning the structure of the slice category $\cat C/X$ tells us whether or not $X$ is terminal. 
\end{example}

\begin{example}
    Building off of \cref{app ex: diffeo}, let $\overline{\NN}$ denote the discrete category with objects $\mathbb{N} \cup \{\infty\}$. There is a dimension functor $\dim\colon \mathbf{Sm}_p \rightarrow \overline{\NN}$ that takes a manifold to its dimension. This property is ${\rm Diff}_p$-recognizable, meaning the dimension of a smooth manifold can be deduced from its group of $p$-fold differentiable automorphisms. This is not surprising, since $\mathbf{Sm}_p$ is $\text{Diff}_p$-reconstructable to begin with.
\end{example}

\begin{proposition}
    Let $D\colon \cat C \rightarrow \cat D$ be a reconstruction problem, and $P\colon \cat C \rightarrow \cat E$ be a property. If $\cat C$ is $D$-reconstructable, then $P$ is $D$-recognizable.
\end{proposition}

In particular, $D$-recognizability implies if $C\cong_D C'$, then $C \cong_P C'$. Moreover, given a morphism $\phi\colon D\to E$ in $\mathbf{Recon}_{\cat C}$, then $\phi_0=Q$ witnesses that $E$ is $D$-recognizable.
Thus, $D$-recognizable properties are generally more interesting when $\cat C$ is not known to be $D$-reconstructable. 

By definition, a recognizable property $P\colon \cat C\to \cat D'$ is itself a reconstruction problem on $\cat C$. This observation gives us a way to compare different reconstruction problems. Write $D \preceq_\cat{C} D'$ if $D'$ is $D$-recognizable. It is straightforward to check that $\preceq_\cat{C}$ defines a preorder on the objects of $\mathbf{Recon}_\cat C$.

\begin{proposition}
    Suppose $D,D'$ are reconstruction problems on a category $\cat C$. If $\cat C$ is $D'$-reconstructable and $D \preceq_{\cat C} D'$, then $\cat C$ is $D$-reconstructable.
\end{proposition}

In other words, if $D \preceq_{\cat C} D'$, then $D$ is an ``easier" reconstruction problem in the sense that if $\cat C$ is $D'$-reconstructable, then it is also $D$-reconstructable.

\begin{definition}
    We say that $D$ and $D'$ are \emph{recognizably equivalent}, denoted $D \approx_\cat C D'$, if both $D \preceq_\cat C D'$ and $D' \preceq_\cat C D$. Then $\approx_\cat C$ is an equivalence relation on $\ob (\mathbf{Recon}_\cat C)$ and the quotient map $\deg_\cat C\colon \ob (\mathbf{Recon}_\cat C)\to \ob (\mathbf{Recon}_\cat C)/\approx_{\cat C}$ assigns $D$ to its \textit{reconstruction degree} $\deg_{\cat C}(D):=[D]$.  
\end{definition}

It is straightforward to check that the preorder $\preceq_\cat C$ defines a partial order on the reconstruction degrees of $\cat C$. We can even prove a stronger result.

\begin{theorem}\label{deg bounded complete lattice}
     The preorder $\preceq_{\cat C}$ gives $\mathbf{Recon}_\cat C / \approx_{\cat C}$ the structure of a bounded complete lattice. 
\end{theorem}
\begin{proof}
    Suppose $\cat C$ is non-empty (if $\cat C$ is empty, the result is trivial). The identity functor $1\colon \cat C \rightarrow \cat C$, and the terminal functor $!\colon \cat C \rightarrow \mathbf{1}$, respectively give the bottom and top degrees, $[1_\cat C]$ and $[!]$. For binary least upper bounds, consider the pushout $\cat U$ of the span $\cat E_0 \xleftarrow{E_0} \cat C_0 \xrightarrow{D_0} \cat D_0$. The functor $U\colon \cat C  \rightarrow \cat U$ given by $U := D_*(E)Di = E_*(D)Ei$, where $i\colon \cat C \rightarrow \cat C_0$ is the functor that takes an object $C$ to its isomorphism class in $\cat C_0$ gives a least upper bound. Spans with set-indexed legs give arbitrary least upper bounds. 
    
    Similarly, for binary greatest lower bounds, let $\cat L$ denote the product $\cat D_0 \times \cat E_0$ in discrete categories with projection maps $\pi_{\cat D_0}$, $\pi_{\cat E_0}$ and let $L_0\colon \cat C_0 \rightarrow \cat L$ denote
the map induced by $D_0$ and $E_0$. The fuctor $L\colon \cat C \rightarrow \cat L$ given by the composition $L = L_0i$ defines a greatest lower bound. Set-indexed products give arbitrary greatest lower bounds. 
\end{proof}


\subsection{Reconstruction problems and covering families}\label{app subsec:recon and CatFam}

In some cases, a reconstruction problem $D\colon \cat C \rightarrow \cat D$ can be realized as a reconstruction problem $\cat C\to W(\cat C)$ for a suitable collection of covering families on $\cat C$. Then $K$-theory provides a tool for analysis of the original problem.

\begin{definition}
    Let $\cat C$ be a category with covering families and $W(\cat C)$ its category of covers. A \emph{covering reconstruction problem} is a reconstruction problem $D\colon \cat{C}^{\cong} \rightarrow W(\cat C)$ that sends an object $C$ to a collection of objects $\{C_i\}_{i \in I}$ that cover $C$, meaning there is a morphism $\{C_i\}_{i\in I}\to \{C\}$ in $W(\cat C)$.
\end{definition}

\begin{definition}
    A category with covering families $\cat C$ \emph{has isomorphism covers} if $\{\phi\colon C' \rightarrow C\}\in W(\cat C)$ for every isomorphism $\phi\colon C' \rightarrow C$ in $\cat C$.
\end{definition}

If $\cat C$ has isomorphism covers, then there is a well-defined map $\iota\colon \cat C_0 \rightarrow K_0(\cat C)$ that sends $[X]$ to itself. If this map is an injection, then there are no two non-isomorphic objects of $\cat C$ with isomorphic covers. Since a covering reconstruction problem picks a cover for every object, this implies every covering reconstruction has a solution.

\begin{theorem}
    Let $\cat C$ is a category with covering families that has isomorphism covers. Suppose the map $\iota: \cat C_0 \rightarrow K_0(\cat C)$ is an injection. 
    Then $\cat C$ is $D$-reconstructable for all covering reconstruction problems $D\colon \cat C^{\cong} \rightarrow W(\cat C)$.
\end{theorem}

The converse is not true in general. The issue is that $D$-reconstructability does not allow an object to be decomposed multiple times, while $K_0$ sees these further decompositions.

We may relate the universal property of $K_0$ (\cref{rmk:UP of K0}) to recognizable properties as follows. If $\cat C$ has isomorphism covers, then a covering invariant $F\colon \ob\cat C\to \cat A$ extends to a functor $F\colon \cat C^{\cong}\to \cat A$, where the Abelian group $\cat A$ is viewed as a discrete category.

\begin{proposition}
    Let $\cat C$ be a category with isomorphism covers and let $D\colon \cat C^{\cong} \rightarrow W(\cat C)$ be a covering reconstruction problem. Then every covering invariant $F\colon \cat C^{\cong} \rightarrow \cat A$ is a $D$-recognizable property.
\end{proposition}
\begin{proof}
Since $F$ is a covering invariant, the map $F'\colon W(\cat C)\to \cat A$ given by $\{C_i\}_{i\in I}\mapsto \sum_{i\in I} F(C_i)$ is well-defined. Hence given any $D\colon \cat C^{\cong}\to W(\cat C)$, we have $F_0 = F'_0\circ D_0$.
\end{proof}

Hence a covering invariant is always $D$-reconstructable. The converse is not true in general because a covering reconstruction problem $D\colon\cat C^{\cong} \rightarrow W(\cat C)$ picks out a single cover for each object of $\cat C$, and so a $D$-recognizable property $P\colon \cat C^{\cong} \rightarrow \cat A$ need only respect this single cover (while a covering invariant must respect all covers). However, in some cases, a version of the converse does hold.

\begin{definition}
    Let $\cat C$ be a category with isomorphism covers and let $D\colon \cat C^{\cong}\to W(\cat C)$ be a covering reconstruction problem. A property $P\colon \cat C^{\cong} \to \cat A$ is \textit{strongly $D$-recognizable} if it factors as\[
    \begin{tikzcd}
        \cat C^{\cong} \ar[r, "D"] \ar[rd, swap, "P"] & W(\cat C) \ar[d, "{P'}"] \\
        & \cat A
    \end{tikzcd}
    \] so that $P'$ is symmetric monoidal.
\end{definition}

Here $\cat A$ is again viewed as a discrete symmetric monoidal category. Recall that the symmetric monoidal structure on $W(\cat C)$ is given by disjoint union, so in particular we must have $P'(\{C\}) = P'(\{C_i\}_{i\in I}) = \sum_{i\in I} P'(\{C_i\})$ whenever $\{C_i\to C\}_{i\in I}$ is a cover.

\begin{proposition}
    If $P$ is strongly $D$-recognizable, then the restriction of $P$ to objects is a covering invariant.
\end{proposition}\begin{proof}
    Let $\{C_i\to C\}_{i\in I}$ be a cover of $C$. As noted above, we have\[
    P'(\{C\}) = P'(\{C_i\}_{i\in I}) = \sum_{i\in I} P'(\{C_i\})
    \] since $P'$ is symmetric monoidal. It thus suffices to show that $P(C) = P'(\{C\})$ for all objects $C$. This is true because there is a morphism $D(C)\to \{C\}$ in $W(\cat C)$, and so since $\cat A$ is discrete we must have $P'(\{C\}) = P'D(C) = P(C)$. 
\end{proof}

\begin{remark}\label{K0 factors}
    If $\cat C$ has isomorphism covers, there are always functors\[
    \begin{tikzcd}
         W(\cat C) \ar[d] & \cat C^{\cong} \ar[dl]  \\
        K_0(\cat C)
    \end{tikzcd}
    \] 
    where $K_0(\cat C)$ is viewed as a discrete symmetric monoidal category and the vertical map is given by $\{C_i\}_{i\in I}\mapsto \sum_{i\in I}[C_i]$. Any group homomorphism $K_0(\cat C)\to \cat A$ determines a functor $W(\cat C)\to \cat A$ (by restricting along the vertical map) and moreover a symmetric monoidal functor $W(\cat C)\to \cat A$ uniquely extends to a group homomorphism $K_0(\cat C)\to \cat A$. Hence we should view the previous two propositions as saying that a functor $\cat C^{\cong}\to \cat A$ factors through $K_0(\cat C)$ if and only if it factors through $W(\cat C)$ in an appropriate way.
\end{remark}

\begin{example}[Edge reconstruction]\label{app ex:edge recon}
The primary example of this paper can be viewed in this abstract framework. Recall that $\Gamma := \mathbf{FinGraph}$ is the category of finite simple graphs, with subgraph inclusions as morphisms. The edge reconstruction problem $\Gamma\to \cat D$ has codomain category equal to the category of finite multisets on $\Gamma$, and takes $G\in \Gamma$ to its edge-deck, $\ED(G)$. 
We can realize this as a covering reconstruction problem by giving $\Gamma$ the following covering families: \begin{itemize}
    \item $\{\varnothing = \varnothing\}_{i \in I}$ for every finite (possibly empty) set $I$;
    \item the singleton $\{f: G \xrightarrow{\cong} G'\}$ for every graph isomorphism $f: G \xrightarrow{\cong} G'$ in $\Gamma$;
    \item for each graph $G$ with at least 4 edges, the collection $\{\iota_C \circ \psi_C : C' \rightarrow G\}_{C \in \ED(G)}$, where $\iota_C: C \hookrightarrow G$ is the inclusion of a card and $\psi_C : C' \xrightarrow{\cong} C$ is a graph isomorphism.
\end{itemize}
The functor $\ED'\colon \Gamma\to W(\Gamma)$ given by
$$\ED'(G) := \begin{cases} 
\ED(G) & \quad \text{if $G$ has $\geq 4$ edges} \\
\{G\} & \quad \text{else}\end{cases}$$
encodes edge reconstruction for graphs with at least 4 edges; the edge reconstruction conjecture says that $\Gamma$ is $\ED'$-reconstructable.
However, since each card in the edge deck of a graph can itself be decomposed by an edge deck, $K_0$ is not just encoding graph reconstruction for this collection of covering families.
\end{example}

\begin{example}[Vertex reconstruction]\label{app ex:vtx recon}
    Similarly, this framework also captures vertex reconstruction for graphs, which says that every graph $G$ with at least $3$ vertices can be reconstructed (up to isomorphism) from its vertex deck $\VD(G)$, the mutiset of vertex deleted subgraphs of $G$. Just as in the previous example, we can endow $\Gamma$ with covering families and construct $\VD'\colon \Gamma \to W_v(\Gamma)$ so that $\Gamma$ is $\VD'$-reconstructable if and only if the vertex reconstruction conjecture is true. 
\end{example}

\begin{example}[Subgroup reconstruction]\label{app ex:subgp recon}
    Let $G$ be a finite group and consider the finite multiset ${\rm Sub}(G):=\{H_i \}_{i \in I}$ 
    of isomorphism classes of proper subgroups of $G$. Is $G$ determined (up to isomorphism) by ${\rm Sub}(G)$? To encode this as a reconstruction problem, let $\cat G:= \mathbf{FinGrp}_{\text{inj}}$ be the category of finite groups and injective group homomorphisms. 
    Give $\cat G$ the following covering families:\begin{itemize}
    \item $\{e=e\}_{i \in I}$ for every finite (possibly empty) set $I$, where $e$ is the trivial group;
    \item the singleton $\{f: G \xrightarrow{\cong} G'\}$ for every group isomorphism $f: G \xrightarrow{\cong} G'$;
    \item the collection $\{\iota_H \circ \psi_H : H' \rightarrow G\}_{H \in {\rm Sub}(G)}$, where $\iota_H: H \hookrightarrow G$ is the inclusion of a non-trivial proper subgroup and $\psi_H : H' \xrightarrow{\cong} H$ is a group isomorphism.
\end{itemize}
Then ${\rm Sub}\colon \cat G\to W(\cat G)$ encodes subgroup reconstructability. However, note that $\cat G$ is not ${\rm Sub}$-reconstructable, since ${\rm Sub}(\ZZ/p\ZZ)=\{e\}$ for any prime $p$. We may instead ask if each \emph{acyclic group} $G$ is fully determined (up to isomorphism) by ${\rm Sub}(G)$. We encode this restricted reconstruction problem by again augmenting our functor ${\rm Sub}: \cat G \rightarrow W(\cat G)$. Define ${\rm Sub}': \cat G \rightarrow W(\cat G)$ by
$${\rm Sub}'(G) = \begin{cases} 
    {\rm Sub}(G) & \quad \text{if } G \text{ is acyclic}\\
    \{G\} & \quad \text{if } G \text{ is cyclic}
\end{cases}.$$

This new functor corresponds to subgroup reconstruction for acyclic groups. 
\end{example}

\subsection{A general example: atomic reconstruction}\label{app sec: atomic}
Here we present a general situation in which a reconstruction problem can be converted into a category with covering families, generalizing the previous three examples. Note that if $\cat C$ is a category with covering families, the category with covering families $W(\cat C)$ is just a certain subcategory of $\FMS(\cat C)$, the category of \textit{finite multisets on $\cat C$}. The objects of $\FMS(\cat C)$ are $I$-indexed multisets $\{C_i\}_{i\in I}$ of objects $C_i\in \cat C$, for some finite set $I$, and a morphism $\{C_i\}_{i\in I}\to \{D_j\}_{j\in J}$ is a set map $f\colon I\to J$ along with $\cat C$-morphisms $f_i: C_i \rightarrow C_{f(i)}$.  

Even if a category $\cat C$ does not have covering families, we can think of choosing a cover of an object $C$ as just specifying a morphism $\{C_i\}_{i\in I}\to \{C\}$ in $\FMS(\cat C)$. We can ask these choices to assemble into a functor $F\colon \cat C^{\cong}\to \vec{\FMS}(\cat C)$, where the codomain category is the arrow category of $\FMS(\cat C)$. We say an object $C$ in $\cat C$ is \emph{fixed} if $F(C) = \{C = C\}$ is the singleton identity map.

\begin{definition}
    Let $\cat C$ be a category with an initial object $\varnothing$. A \emph{patching} on $\cat C$ is a functor $P\colon \cat C^{\cong}\to \vec{\FMS}(\cat C)$ such that (i) $\varnothing$ is fixed and (ii) if $C$ is not fixed, then $P(C) = \{C_i \rightarrow C \}_{i \in I}$ satisfies $|I| \geq 2$ and $C_i \not \cong C$ for all $i \in I$.
\end{definition}

\begin{definition}
A patching $P$ on $\cat C$ has a corresponding \emph{decomposition} $D_P\colon \cat C^{\cong}\to \FMS(\cat C)$ given by the composition $D_P := \dom \circ P$. When the patching is clear, we will omit the subscript on $D_P$.
\end{definition}

That is, for an object $C$ in $\cat C$ with patching $P(C) = \{f_i : C_i\to C \}_{i \in I}$, the decomposition of $C$ is the multiset of objects $D(C) := \{C_i \}_{i \in I}$. The decomposition $D(C)$ is either exactly $\{C\}$ (when $C$ is fixed by the patching $P$), or $C$ is decomposed into at least two objects. These objects are non-trivial in the sense that they are not isomorphic to $C$. Intuitively, we think of the decomposition $D$ as either fixing $C$, or breaking $C$ into a multiset of ``smaller" pieces. The patching $P(C)$ tells us how to assemble the pieces $D(C)$ back into $C$. The decomposition $D$ will be the reconstruction problem we are interested in. 

\begin{remark}
It is perhaps illuminating to explicate the patchings and decompositions in \cref{app ex:edge recon}, \cref{app ex:vtx recon}, and \cref{app ex:subgp recon}. For \cref{app ex:edge recon}, the patching assigns a graph (with at least 4 edges) $G$ to $\{\iota_C\colon C\hookrightarrow G\}_{C\in \ED(G)}$, where $\iota_C$ is the inclusion of a card as a specific subgraph; the decomposition is then the edge deck as a multiset. Note that the difference between the patching and the decomposition is whether or not we keep track of the specific inclusions.  We leave the other two examples as exercises for the reader.
\end{remark}

We now outline a general setting for reconstruction problems, wherein ``molecules'' are decomposed into fixed ``atomic'' objects. For the remainder of the section, we fix a small category $\cat C$ with a unique initial object $\varnothing$ such that there are no morphisms $X \rightarrow \varnothing$ for non-initial $X$. 

\begin{definition}\label{rec setting}
     Suppose the objects of $\cat C$ are partitioned as $\ob \cat C = \{ \varnothing \} \amalg \mathcal{A} \amalg \mathcal{M}$; the objects in $\mathcal{A}$ are called \emph{atoms} while the objects in $\mathcal{M}$ are called \emph{molecules}. Let $P\colon \cat C^{\cong} \to \vec{\FMS}(\cat C)$ be a patching on $\cat C$ with corresponding decomposition $D$. The triple $(\mathcal A, \mathcal M, P)$ is an \emph{atomic reconstruction setting} if the following conditions hold:
    \begin{enumerate}
        \item $\mathcal{A}$ is non-empty;
        
        \item $\mathcal A$ and $\mathcal M$ are closed under isomorphism;
        \item every $A \in \mathcal A$ is fixed by $P$;
        \item every molecule $M \in \mathcal{M}$ decomposes $D(M) = \{A_i \}_{i \in I}$ into atoms $A_i \in \mathcal{A}$.
    \end{enumerate}
\end{definition}

From such a reconstruction setting, we get a category with covering families (with isomorphism covers) for free. Each of the requirements of \cref{defn:CatFam} is straightforward to check.

\begin{proposition}\label{ex:rec setting covering fams 2}
    Let $(\mathcal A, \mathcal M, P)$ be an atomic reconstruction setting on $\cat C$. Consider the following collection of covering families on $\cat C$:
    \begin{enumerate}
        \item $\{\varnothing = \varnothing\}_{i \in I}$ for every finite (possibly empty) set $I$;
        \item the singleton $\{f\colon C\xrightarrow{\cong} C'\}$ for every isomorphism $f\colon C\xrightarrow{\cong} C'$ in $\cat C$;
        \item the collection $\{F_j \circ \psi_j \colon A_j' \rightarrow M\}$, where $P(M) = \{F_j \colon A_j \rightarrow M\}_{j\in J}$ is the patching for a molecule $M \in \mathcal M$ and $\psi_j\colon A_j' \xrightarrow{\cong} A_j$ is an isomorphism for each $j \in J$.
    \end{enumerate}
    This collection gives $\cat C$ the structure of a category with covering families with distinguished object $\varnothing$. The functor $D\colon \cat{C}^{\cong} \rightarrow W(\cat C)$ given by $C \mapsto D(C)$ is a covering reconstruction problem.
\end{proposition}

If $(\mathcal A, \mathcal M, P)$ is an atomic reconstruction setting on $\cat C$, so $\cat C$ is a category with covering families via \cref{ex:rec setting covering fams 2}, we have
$$K_0(\cat C) \cong \ZZ[\cat C_0] / \sim$$
where $\sim$ is generated by $C \sim C'$ if $D(C) \cong D(C')$. Since we are in an atomic reconstruction setting, each object can only be decomposed by $D$ at most once, implying the following proposition.

\begin{proposition}
    Let $(\mathcal A, \mathcal M, P)$ be an atomic reconstruction setting on $\cat C$. Then $\cat C$ is $D$-reconstructable if and only if $\iota \colon \cat C_0 \rightarrow K_0(\cat C) $ is injective.
\end{proposition}

While ``molecules'' and ``atoms'' are not required to correspond to some intuitive notion of size, they do in most cases we have considered. In particular, in order to modify the examples at the end of \cref{app subsec:recon and CatFam} where there is no clear divide of the objects into atoms and molecules, we need to introduce a notion of size.

\begin{definition}
    Let $P\colon \cat C^{\cong} \rightarrow \vec{FMS}(\cat C)$ be a patching with corresponding decomposition $D$. A \emph{size map} for $D$ is a functor $\abs{-} \colon \cat C^{\cong}\to \NN$, with $\NN$ viewed as a discrete category, such that:
    \begin{itemize}
        \item for all non-initial objects $C$, we have $|\varnothing| < |C|$; 
        \item $\abs{{\rm im}(\abs{-})} \geq 3$;
        \item for all unfixed objects $C$ with $D(C) = \{C_i\}_{i \in I}$, we have $|C_i| < |C|$ for all $i \in I$.
    \end{itemize}
\end{definition}


In other words, $\abs{-}$ assigns every object $C$ of $\cat C$ to some size $|C|=n$, there are at least two different sizes of non-initial object, and $\abs{-}$ is constant on isomorphism classes. If $C$ is not fixed by $\abs{-}$, the pieces in the decomposition $D(C)$ are always of a smaller size than $C$. In \cref{app ex:edge recon} and \cref{app ex:vtx recon}, $P$ is given by number of edges and vertices, respectively; in \cref{app ex:subgp recon}, $P$ is the order of the group.

\begin{remark}
    Requiring size to be valued in $\NN$ is not strictly necessary for what follows. One could generalize to functors into arbitrary discrete categories that satisfy certain conditions, however such generality is unnecessary for our purposes. 
\end{remark}

\begin{definition}
    Let $P$ be a patching with decomposition $D$ and let $\abs{-}\colon \cat C^{\cong} \rightarrow \NN$ be a size map. For $n \in \NN$, set $\mathcal{A}_n := \{C \in \ob \cat C \mid |C| \neq n \text{ or } C \text{ is fixed and non-initial} \}$ and $\mathcal{M}_n := \{C \in \ob \cat C \mid |C| = n \text{ and $C$ is not fixed} \}$. Define $P_n\colon \cat C^{\cong} \rightarrow \vec{\FMS}(\cat C)$ by
$$P_n(C) := \begin{cases} 
    P(C) & \quad \text{ if } C \in \mathcal M_n, \\
    \{C = C \} & \quad \text{ if } C \not \in \mathcal{M}_n.
\end {cases}$$
\end{definition}

\begin{proposition}
The triple $(\mathcal A_n, \mathcal M_n, P_n)$ is an atomic reconstruction setting on $\cat C$ for all $n$.
\end{proposition}
\begin{proof}
    It is easy to verify that $\ob \cat C = \{\varnothing\} \amalg \mathcal{A}_n \amalg \mathcal{M}_n$, that $P_n$ is a patching, that $\mathcal{A}_n$ and $\mathcal{M}_n$ are closed under isomorphism, and that each $A \in \mathcal{A}_n$ is fixed by $P_n$. It is also easy to check that $\mathcal{A}_n$ is non-empty from the definition of a size map. Finally let $D_n := \dom \circ P_n$ be the decomposition corresponding to $P_n$. For $M \in \mathcal{M}_n$ with decomposition $D_n(M) = \{X_i\}_{i \in I}$, since $|X_i| < |M|$ for all $i \in I$, it follows that each $X_i \in \mathcal{A}_n$. Thus we have an atomic reconstruction setting. 
\end{proof}

The idea is that we can study the decomposition $D$ as a reconstruction problem on $\cat C^{\cong}$ by instead studying the collection of atomic reconstruction settings $\{(\mathcal{A}_n, \mathcal{M}_n, D_n) \}_{n \in \NN}$. However, we need to ensure that no information is gained by applying the size map, which essentially amounts to assuming that size is a recognizable property.

\begin{theorem}\label{size reconstruction}
    If a size map $\abs{-}$ is $D$-recognizable, then $\cat C^{\cong}$ is $D$-reconstructable if and only if $\cat C^{\cong}$ is $D_n$-reconstructable for all $n \in \NN$. 
\end{theorem}

\begin{proof}
    Suppose $\cat C^{\cong}$ is $D$-reconstructable, and suppose $X \cong_{D_n} Y$ for some $n\in \NN$. Then we have either $X,Y \in \mathcal{M}_n$ or $X,Y \not \in \mathcal{M}_n$. In the first case, we get that $X \cong_{D} Y$, and consequently $X \cong Y$ since $\cat C^{\cong}$ is $D$-reconstructable. In the second case, $X$ and $Y$ are fixed by $P_n$, so $X \cong Y$. 

    Conversely, suppose that $\cat C^{\cong}$ is $D_n$-reconstructable for all $n \in \NN$, and suppose $X \cong_{D} Y$. Then $X$ and $Y$ are either both fixed or both not fixed by $P$. If both are fixed, then $X \cong Y$, so suppose neither are fixed. Then since $\abs{-}$ is $D$-recognizable, $|X| = |Y|=n$ for some $n$ and so $X,Y \in \mathcal{M}_n$ and $X\cong_{D_n} Y$. Since $\cat C^{\cong}$ is $D_n$-reconstructable, this implies $X\cong Y$.
\end{proof}

\begin{remark}
    Note that this theorem would not hold if  $P$ were a patching, which is why we require $P$ to be a patching in the definition of an atomic reconstruction setting. 
\end{remark}

\begin{example}
Continuing from \cref{app ex:edge recon}, let $\abs{-}\colon \Gamma \rightarrow \NN$ send a graph to its number of edges. Then $|G|$ is recognizable from $\ED(G)$, so we may break up the reconstruction problem by number of edges. For each $n \in \NN$, let $\Gamma_n$ be $\Gamma$ with the following covering families:\begin{itemize}
    \item $\{\varnothing = \varnothing\}_{i \in I}$ for every finite (possibly empty) set $I$;
    \item the singleton $\{f: G \xrightarrow{\cong} G'\}$ for every graph isomorphism $f: G \xrightarrow{\cong} G'$ in $\Gamma$;
    \item for each graph $G$ with exactly $n$ edges and at least $4$ edges, the collection $\{\iota_C \circ \psi_C : C' \rightarrow G\}_{C \in \ED(G)}$, where $\iota_C: C \hookrightarrow G$ is the inclusion of a card and $\psi_C : C' \xrightarrow{\cong} C$ is a graph isomorphism.
\end{itemize}
The functor $\ED'_n \colon \Gamma \rightarrow W(\Gamma_n)$ encodes breaking $n$-edge graphs into their $n-1$-edge subgraphs, and $\Gamma$ is $\ED'$-reconstructable if and only if $\Gamma$ is $\ED'_n$-reconstructable for all $n$ (note this is trivial for $n \leq 3$). Further, $\Gamma$ is $\ED'_n$-reconstructable if and only if the map $\iota \colon \Gamma_0 \rightarrow K_0(\Gamma_n)$ is an injection, which has been confirmed for small $n$.
\end{example}

\begin{example}[Vertex reconstruction continued]
    Just as number-of-edges gives a recognizable size map for edge reconstruction, number-of-vertices gives a recognizable size map for vertex reconstruction and we may make analogous definitions as above.
\end{example}

\begin{example}
    Continuing from \cref{app ex:subgp recon}, let $\abs{-}: \cat G \rightarrow \mathbb{N}$ send a finite group to its order. While $\abs{-}$ is not ${\rm Sub}$-recognizable, it is ${\rm Sub}'$-recognizable by a straightforward argument using the first Sylow theorem. 
    In some cases, $\cat G_n$ is ${\rm Sub}'$-reconstructable and in other cases it is not; for example, $\cat G_4$ is ${\rm Sub}'$-reconstructable (the Klein $4$ group and $\ZZ/4\ZZ$ have different subgroups) but $\cat G_{16}$ is not. One can check that $\ZZ/2\ZZ \times \ZZ/8\ZZ$ has the same subgroup lattice as the modular group on 16 elements \cite{groupsMO}--- hence these groups are not reconstructable from their lattice of subgroups, much less the multiset of isomorphism classes of subgroups. By \cref{size reconstruction}, $\cat G$ is not ${\rm Sub}'$-reconstructable. 
\end{example}
\bibliographystyle{alpha}
\bibliography{references}

\end{document}